\documentclass{amsart}

\usepackage{amsxtra}

\usepackage{amssymb}
\usepackage{amsmath}
\usepackage{euscript}
\usepackage{pstricks,pst-node,pst-text,pst-3d,pst-plot}

\usepackage{graphics}
\usepackage{epsfig}

\newcommand{\Z}{\mathbb Z}
\newcommand{\Q}{\mathbb Q}
\newcommand{\N}{\mathbb N}
\newcommand{\F}{\mathbb F}
\newcommand{\R}{\mathbb R}

\newcommand{\Cal}{\EuScript}

\newcommand{\XOR}{\mathbin{\mathsf{XOR}}}
\newcommand{\OR}{\mathbin{\mathsf{OR}}}
\newcommand{\AND}{\mathbin{\mathsf{AND}}}

\newcommand{\NOT}{\operatorname{\mathsf{NOT}}}
\newcommand{\ord}{\operatorname{ord}}

\newcommand{\md}{\mathbin{\mathsf{mod}}}

\renewcommand{\:}{\colon}
\renewcommand{\>}{\rightarrow}

\newtheorem{thm}{Theorem}[section] 

\newtheorem{cor}[thm]{Corollary}

\theoremstyle{definition}
\newtheorem{defn}[thm]{Definition}
\newtheorem*{defn*}{Definition}
\newtheorem{note}[thm]{Note}
\newtheorem{exmp}[thm]{Example}

\title[The Non-Archimedean Theory of Discrete Systems]{The Non-Archimedean Theory\\ of Discrete Systems}\thanks{Supported in
parts by Russian Foundation for Basic Research grant No 12-01-00680-a and by Chinese Academy of Sciences visiting professorship for senior international
scientists grant No  2009G2-11
}
\author{Vladimir Anashin}
\address{
Faculty of Computational Mathematics and Cybernetics\\
Lomonosov Moscow State University\\
Leninskie Gory, 1\\
119991 Moscow\\
Russia}
\email{vs-anashin@yandex.ru; vladimir.anashin@u-picardie.fr}

\begin{document}



\begin{abstract}
In the paper, we study behaviour of discrete dynamical systems (automata)
w.r.t. transitivity;
that is, speaking loosely, we consider 
how diverse may be behaviour of the system w.r.t.  variety
of  word transformations  
performed by the system: We call a system completely transitive if, given
arbitrary pair $a,b$ of finite words that have equal lengths, the system $\mathfrak A$, while evolution during (discrete) time,
at a certain moment   transforms $a$ into $b$.
To every system  $\mathfrak A$, we put into a correspondence a family $\Cal
F_{\mathfrak
A}$ of continuous mappings of a suitable non-Archimedean metric space and show that the
system is completely transitive if and only if the family $\Cal
F_{\mathfrak
A}$ is ergodic w.r.t. the Haar measure; then we find easy-to-verify
conditions the
system must satisfy to be completely transitive. The theory can be applied
to analyse behaviour of straight-line computer programs (in particular, 
pseudo-random
number generators that are used in cryptography and simulations) since basic CPU instructions (both numerical and logical) can be considered as continuous mappings of a (non-Archimedean)
metric space $\Z_2$ of 2-adic integers.
\end{abstract}

\maketitle

\section{Introduction}
According to the  most general definition of a system (see e.g. \cite{KaFaArbib}),
by a discrete system (further --- a system) we understand a stationary dynamical system with a
discrete time $\N_0=\{0,1,2,\ldots\}$; that is, a 5-tuple $\mathfrak A=\langle\Cal I,\Cal S,\Cal O,S,O\rangle$
where $\mathcal I$ is a non-empty finite set, the \emph{input alphabet}; $\Cal O$  is a non-empty finite set, the \emph{output alphabet}; $\Cal S$ is a non-empty (possibly, infinite) set of \emph{states}; 
$S\colon\Cal I\times\Cal S\to \Cal S$ is  a \emph{state transition function}; $O\colon\Cal I\times\Cal S\to \Cal O$ is an \emph{output function}. 
Note that in literature systems are also called (synchronous) automata; however, in order to avoid misunderstanding,
in the paper only \emph{initial} automata are so referred.
Remind that the \emph{initial automaton} $\mathfrak A(s_0)=\langle\Cal I,\Cal S,\Cal O,S,O, s_0\rangle$ is a discrete system
$\mathfrak A$ where one state $s_0\in\Cal S$ is fixed; $s_0$ is called the \emph{initial
state}. At the moment $n=0$ the system  $\mathfrak A(s_0)$ is at the state
$s_0$; once feeded by the input symbol $\chi_0\in\Cal I$, the system outputs the
symbol $\xi_0=O(\chi_0,s_o)\in\Cal O$ and reaches the state
$s_1=S(\chi_0,s_0)\in\Cal S$; then the system is feeded by the next input symbol
$\chi_1\in\Cal I$ and repeats the routine. We
stress that the definition of  the automaton $\mathfrak A(s_0)$ is nearly the same as the one of
\emph{Mealy automaton} (see e.g. \cite{Bra,VHDL}) (or of a `letter' 
\emph{transducer}, see e.g.
\cite{Allouche-Shall,Grigorch_auto2_eng}), with the only important difference: the automata
$\mathfrak A(s_0)$
we consider in the paper are 
 \emph{not necessarily finite}; that is, the set of states
$\Cal S$ of $\mathfrak A(s_0)$ may be infinite.
Furthermore, throughout the paper  we assume that there exists a state
$s_0\in \Cal S$ such that all the states of the system $\mathfrak A$  are \emph{reachable}
from $s_0$; 
that is, given $s\in\Cal S$, there exists
input word $w$ over alphabet $\Cal I$ such that after the word $w$ has been
feeded to the automaton $\mathfrak A(s_0)$, the automaton reaches the state $s$. To  
 the system
$\mathfrak A$ we put into a correspondence the family $\Cal F(\mathfrak
A)$ of all automata $\mathfrak
A(s)=\langle\Cal I,\Cal S,\Cal O,S,O, s\rangle$, $s\in\Cal S$.  
For better
exposition, 
throughout the paper we assume that both alphabets $\Cal I$ and $\Cal O$
are $p$-element alphabets with $p$  prime:  
$\Cal I=\Cal O=\{0,1,\ldots,p-1\}=\F_p$; so further 
the word `automaton' stands for initial automaton with input/output alphabets
$\F_p$. A typical example of an automaton of that sort is 
the  \emph{2-adic adding machine}
$\mathfrak O(1)=\langle\F_2,\F_2,\F_2,S,O,1\rangle$, where
$S(\chi,s)\equiv \chi s\pmod
2$,  
$O(\chi,s)\equiv\chi+s\pmod
2$ for $s\in\Cal S=\F_2$, $\chi\in\Cal I=\F_2$.

Automata often are represented by \emph{Moore diagrams}.
Moore  diagram of the automaton $\mathfrak A(s_0)=\langle\Cal I,\Cal S,\Cal O,S,O,s_0\rangle$
is
a directed labeled graph whose vertices are identified with the states of the automaton and for every $s\in\Cal S$ and $r\in\Cal I$ the diagram has an arrow from $s$ to $S(r,s)$ labeled by the pair $(r,O(r,s))$. Figure
\ref{fig:Moore} is an
example of Moore diagram.
\begin{figure}[h]
\begin{quote}\psset{unit=0.5cm}
 \begin{pspicture}(0,1.5)(24,7)
\pnode(8.7,3.8){S1-o}
\pnode(8.6,4.3){S1-i}
\pnode(15.4,4.4){S0-o}
\pnode(14.8,4.5){S0-i}
\pnode(15.4,3.6){S00-o}
\pnode(14.8,3.5){S00-i}
\nccurve[angleA=150,angleB=90,ncurv=-12]{<-}{S00-o}{S00-i}
\nccurve[angleA=210,angleB=270,ncurv=-12]{<-}{S0-o}{S0-i}
\nccurve[angleA=210,angleB=150,ncurv=15]{->}{S1-o}{S1-i}
\pscircle[fillstyle=solid,fillcolor=yellow,linewidth=1pt](9,4){.5}
\pscircle[fillstyle=solid,fillcolor=yellow,linewidth=1pt](15,4){.5}
  \psline{->}(9.5,4)(14.5,4)
 \uput{1}[90](9,2.8){$1$}
  \uput{1}[90](15,2.8){$0$}
  \uput{1}[90](12,3.3){$(0,1)$}
  \uput{1}[90](17.5,.3){$(1,1)$}
  \uput{1}[90](17.5,5.3){$(0,0)$}
  \uput{1}[90](5.2,2.6){$(1,0)$}
\end{pspicture}
\end{quote}
\caption{Moore diagram of the 2-adic adding machine.}
\label{fig:Moore}
\end{figure}
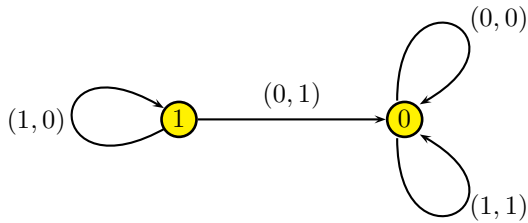

Given an automaton $\mathfrak
A(s)=\langle\F_p,\Cal S,\F_p,S,O, s\rangle\in\Cal F(\mathfrak A)$, the
automaton transforms input words of length $n$ into output words of length $n$; that is, $\mathfrak
A(s)$ maps the set $W_n$ of all words of length $n$ into $W_n$; we denote
 corresponding mapping via $f_{n,\mathfrak A(s)}$. It is clear now that behaviour of the system $\mathfrak
A$ can be described in terms of the mappings $f_{n,\mathfrak A(s)}$ for all $s\in\Cal S$
and all $n\in\N=\{1,2,3,\ldots\}$. As all states of the system $\mathfrak A$ are
reachable from the state $s_0$, it suffices to study 
only the mappings $f_{n,\mathfrak A(s_0)}$ for all $n\in\N$.
Now we remind the notion of transitivity:
\begin{defn}[Transitivity of a family of mappings]
\label{def:trans}
A \emph{family $\Cal F$ of mappings} of a finite non-empty set $M$ into $M$ is called \emph{transitive}
whenever given a pair $(a,b)\in M\times M$, there exists $f\in\Cal F$ such that $f(a)=b$.
\end{defn} 
It is clear that once $M$ contains
more than one element, a family that consists of a single mapping $f\:M\>M$ cannot
be transitive in the meaning of Definition \ref{def:trans}; that is why the transitivity of a single mapping is defined as follows:
\begin{defn}[Transitivity of a single mapping]
\label{def:trans-1}
A \emph{mapping} $f\:M\>M$, where $M$ is a finite non-empty set, is called \emph{transitive} if it $f$ cyclically  permutes elements of $M$.
\end{defn}
In other words,  a single mapping $f\:M\>M$ is transitive if and only if
the family $\{e,f, f^{2}=f\ast
f, f^3=f\ast f\ast f,\ldots\}$
is transitive in the meaning of the Definition \ref {def:trans} (here $e$ stands for identity
mapping, $\ast$ for composition of mappings). Note that a transitive mapping
is necessarily bijective but generally not vice versa.

Now we introduce the main notions of the paper.
\begin{defn}[Automata transitivity]
\label{def:auto-trans}
The automaton $\mathfrak A(s_0)$
(equivalently, the system $\mathfrak A$) is said to be
\begin{itemize}
\item \emph{$n$-word transitive}, if the mapping $f_{n,\mathfrak A(s_0)}$ is transitive 
on the set $W_n$ of all words of length $n$;
\item\emph{word transitive}, if $\mathfrak A(s_0)$ is $n$-word transitive for all $n\in\N=\{1,2,3,\ldots\}$;
\item\emph{completely transitive}, if  for every $n\in\N$, the family $f_{n,{\mathfrak A(s)}}$,
$s\in \Cal S$, is transitive on $W_n$;
\item\emph{absolutely transitive}, if for every $s\in \Cal S$ the automaton ${\mathfrak A}(s)$ is completely transitive; that is, if for every $n\in\N$
the family $f_{n,{\mathfrak A(t)}}$,
$t\in \Cal S_{\mathfrak A(s)}$, is transitive on $W_n$, where $\Cal S_{\mathfrak A(s)}$ is the set
of all reachable states of the automaton $\mathfrak A(s)$.
\end{itemize}
\end{defn}
The transitivity properties may be defined in equivalent way, in terms of
words; this way is more common in automata theory. We remand some notions related
to words beforehand.

Given a non-empty alphabet $\Cal A$, its elements are called \emph{symbols},
or \emph{letters}. By the definition, a \emph{word of length $n$ over alphabet $\Cal A$}\index{word
(over an alphabet)}\index{word
(over an alphabet)!-- of length $n$} is a finite string (stretching from
right to left)
$\alpha_{n-1}\cdots\alpha_1\alpha_0$, where $\alpha_{n-1},\ldots,\alpha_1,\alpha_0\in\Cal
A$. The \emph{empty word} is a sequence
of length 0, that is, the one that contains no symbols. Hereinafter the length
of the word $w$ is denoted via $\Lambda(w)$. Given a word 
$w=\alpha_{n-1}\cdots\alpha_1\alpha_0$,
any word $v=\alpha_{k-1}\cdots\alpha_1\alpha_0$, $n\ge k\ge 1$, is called a 
\emph{prefix} (or, an \emph{initial subword})
of the word $w$, any word $u=\alpha_{n-1}\cdots\alpha_{i+1}\alpha_i$,
$0\le i\le n-1$ is called a \emph{suffix} of the word $w$, and any word 
$\alpha_{k}\cdots\alpha_{i+1}\alpha_i$, $n-1\ge k\ge i\ge 0$, is called
a \emph{subword} of the word $w$.    Given words $a=\alpha_{n-1}\cdots\alpha_1\alpha_0$
and $b=\beta_{k-1}\cdots\beta_1\beta_0$, the
\emph{concatenation} $a\circ b$ is the following word (of length $n+k$):
\[
a\circ b=\alpha_{n-1}\cdots\alpha_1\alpha_0\beta_{k-1}\cdots\beta_1\beta_0.
\]  
\begin{defn}[Automata transitivity, equivalent]
\label{def:auto-trans-e}
~
\begin{enumerate} 
\item The \emph{word transitivity} means that given two finite words $w$, $w^\prime$ whose lengths are equal one to another,
$\Lambda(w)=\Lambda(w^\prime)=n$, the word $w$ can be transformed into $w^\prime$ by a sequential composition
of a sufficient number of copies of $\mathfrak A(s_0)$:
\begin{figure}[h]
\begin{quote}\psset{unit=0.5cm}
 \begin{pspicture}(0,10)(24,12)
  \psline[linecolor=black]{<-}(20,11)(18,11)
  \psline[linecolor=black]{<-}(6,11)(4,11)
  \psline[linecolor=black](4,11.3)(4,10.7)
  
  \psframe[fillstyle=solid,fillcolor=yellow,linecolor=black,linewidth=2pt](6,10)(10,12)
  \psframe[fillstyle=solid,fillcolor=yellow,linecolor=black,linewidth=2pt](14,10)(18,12)
  
  \uput{0}[90](8,10.7){$\mathfrak A$}
  \uput{0}[90](16,10.7){$\mathfrak A$}
  
  \uput{0}[90](5,9){$\underbrace{}_{ w}$}
  \uput{0}[90](19,9){$\underbrace{}_{ w^\prime}$}
  \uput{0}[90](12,10.7){$\cdots\cdots\cdots$}
  
 \end{pspicture}
\end{quote}

\end{figure}

\item The \emph{complete transitivity} means that given finite words $w$, $w^\prime$ such that
$\Lambda(w)=\Lambda(w^\prime)$, there exists a finite word  $y$ (may be of length other than that
of $w$ and $w^\prime$) such that the automaton $\mathfrak A(s_0)$  transforms the input word $w\circ y$ (with the prefix $y$) 
to the output word $w^\prime \circ y^\prime$ that has a suffix $w^\prime$:
\begin{figure}[h]

\begin{quote}\psset{unit=0.5cm}
 \begin{pspicture}(0,10)(24,12)

  \psline[linecolor=black]{<-}(20,11)(14,11)
  \psline[linecolor=black]{<-}(10,11)(4,11)
  
  \psline[linecolor=black](6,11.3)(6,10.7)
  \psline[linecolor=black](4,11.3)(4,10.7)
  \psline[linecolor=black](16,11.3)(16,10.7)
  
  \psframe[fillstyle=solid,fillcolor=yellow,linecolor=black,linewidth=2pt](10,10)(14,12)
  
  \uput{0}[90](12,10.7){$\mathfrak A$}
  
  \uput{0}[90](8,9){$\underbrace{\ast\cdots\cdots\cdots\ast}_y$}
  \uput{0}[90](5,9){$\underbrace{}_{ w}$}
  \uput{0}[90](15,9){$\underbrace{}_{ w^\prime}$}
  
 \end{pspicture}
\end{quote}

\end{figure} 
\item The \emph{absolute transitivity} means that given finite words $x$, $w$, $w^\prime$ such that
$\Lambda(w)=\Lambda(w^\prime)$ (may be $\Lambda(x)\ne\Lambda(w)$), there exists a finite word $y$ such that the automaton $\mathfrak A(s_0)$  transforms the input word $w\circ
y\circ x$ to the output word $w^\prime\circ y^\prime\circ x^\prime$:
\begin{figure}[h]

\begin{quote}\psset{unit=0.5cm}
 \begin{pspicture}(0,10)(24,12)

  \psline[linecolor=black]{<-}(20,11)(14,11)
  \psline[linecolor=black]{<-}(10,11)(4,11)
  \psline[linecolor=black](8,11.3)(8,10.7)
  \psline[linecolor=black](6,11.3)(6,10.7)
  \psline[linecolor=black](4,11.3)(4,10.7)
  \psline[linecolor=black](16,11.3)(16,10.7)

  \psframe[fillstyle=solid,fillcolor=yellow,linecolor=black,linewidth=2pt](10,10)(14,12)
  \uput{0}[90](12,10.7){$\mathfrak A$}
  \uput{0}[90](9,9){$\underbrace{}_{ x}$}
  \uput{0}[90](7,9){$\underbrace{\ast\cdots\ast}_y$}
  \uput{0}[90](5,9){$\underbrace{}_{ w}$}
  \uput{0}[90](15,9){$\underbrace{}_{ w^\prime}$}

 \end{pspicture}
\end{quote}

\end{figure}

\end{enumerate}
\end{defn}
\begin{exmp}[Word transitive automaton]
\label{exm:odo}
The 2-adic adding machine $\mathfrak O(1)$, which was introduced above, is word transitive:
 It is clear that if one treats an $n$-bit word
as a base-2 expansion of a non-negative integer $w$ then 
$f_{n,\mathfrak O(1)}(w)\equiv w+1\pmod{2^n}$,
$n=1,2,3,\ldots$; therefore $f_{n,\mathfrak O(1)}^i(w)\equiv w+i\pmod{2^n}$
for all $i\in\N_0=\{0,1,2,\ldots\}$ which means that
$f$ is transitive on the set $W_n$ of all $n$-bit words, cf. Definition \ref{def:auto-trans},
Definition \ref{def:trans-1} and Definition \ref{def:auto-trans-e}(i).
\end{exmp}
Note that  the 2-adic adding machine $\mathfrak O(1)$ is not completely transitive as  given $n\in\N=\{1,2,3,\ldots\}$,
the corresponding
family consists of the following two mappings: $f_{n,\mathfrak O(1)}(w)\equiv
w+1\pmod{2^n}$ and $f_{n,\mathfrak O(0)}(w)\equiv
w\pmod{2^n}$; so none of the mappings maps the two-bit word 00 (which is
a  base-2 expansion of 0) to the two-bit word 10 (which is a base-2 expansion
of 2).
\begin{exmp}[Absolutely transitive automaton]
\label{exm:const}
Let $(\alpha_i)_{i=0}^\infty=\alpha_0,\alpha_1,\ldots$ be an infinite binary sequence such that every
binary pattern $\beta_1\cdots\beta_n$ occurs in the sequence $(\alpha_i)_{i=0}^\infty$
(whence, occurs infinitely many times);
that is, given $n\in\N=\{1,2,\ldots\}$ and $\beta_1,\ldots,\beta_n\in\F_2$,
the following equalities $\alpha_i=\beta_1,\alpha_{i+1}=\beta_2,\ldots,\alpha_{i+n-1}=\beta_n$ hold simultaneously
for some (equivalently, for infinitely many) $i\in\N_0=\{0,1,2,\ldots\}$. Then the following automaton
$\mathfrak  C(0)$ is absolutely transitive:
$\mathfrak C(0)=\langle\F_2,\N_0,\F_2,S,O,0\rangle$, where $S(\chi,s)=s+1$,
$O(\chi,s)=\alpha_s$ for $s\in\Cal S=\N_0$, $\chi\in\Cal I=\F_2$.

Indeed, given an $n$-bit word $w$, we see that $f_{n,\mathfrak C(s)}(w)=\alpha_{s+n-1}\cdots\alpha_s$
for every $s\in\Cal S=\N_0$ which by Definition \ref{def:auto-trans-e}(iii) (or, equivalently, by Definition
\ref{def:auto-trans})
implies absolute transitivity
of the automaton $\mathfrak C(0)$ due to the choice of the sequence $(\alpha_i)_{i=0}^\infty$.
\end{exmp}
Note also that the automaton $\mathfrak C(0)$ is $n$-word transitive for no
$n\in\N=\{1,2,3,\ldots\}$ as $f_{n,\mathfrak C(0)}$ is not bijective on $W_n$, cf. Definitions \ref{def:trans-1} and \ref{def:auto-trans}. 

The goal of the paper is to present techniques to determine whether a  system $\mathfrak A$ is word transitive, or completely
transitive, or absolutely transitive. For this purpose, we study how the automaton $\mathfrak A(s_0)$ acts on
\emph{infinite} words over alphabet $\F_p$. The latter words are considered
as $p$-adic integers, and the corresponding transformation turns out to be
a continuous transformation on the space of $p$-adic integers $\Z_p$. We remind  main notions of $p$-adic
analysis in the next section where we describe our approach, first formally
and then less formally.

We note that the $p$-adic approach (and wider the non-Archimedean one) has already been successfully applied to automata
theory. Seemingly the paper  \cite{Lunts} is the first one where  the $p$-adic techniques is applied to
study automata functions; the paper  deals with linearity
conditions of automata mappings. For application of the non-Archimedean methods to automata
and formal
languages  see expository paper \cite{Pin_p-adic_auto} and references therein;
for applications to automata and group theory see \cite{Grigorch_auto,Grigorch_auto2_eng}.
In \cite{Vuillem_circ,Vuillem_DigNum,Vuillem_fin} the 2-adic methods are
used to study binary automata, in particular, to obtain the finiteness criterion for these automata. In monograph \cite{AnKhr} the $p$-adic ergodic theory
is studied (see numerous references therein) aiming at applications to computer
science  and cryptography (in particular, to automata theory, to pseudorandom
number generation and to stream cipher design) as well as to applications
in other areas like quantum theory, cognitive sciences and genetics.

As for mathematical techniques used in the paper, these are somewhat complex: to study ergodic
properties of families 
of automata functions related to a given discrete system, we combine
$p$-adic methods, methods of real analysis and methods from automata theory.
 
The paper is organized as follows:
\begin{itemize}
\item In Section \ref{sec:p-adic} we remind basic notions of $p$-adic analysis
and show that automata functions (the transformations of infinite words performed
by automata) are continuous (actually, 1-Lipschitz) functions
w.r.t. the $p$-adic metric. In particular, we mention that basic computer instructions, both arithmetic (like addition, subtraction and multiplication of
integers) and bitwise logical (like bitwise conjunction, disjunction, negation
and exclusive `or'), as well as some other (like shifts towards higher order bits and
masking) are continuous w.r.t. 2-adic metric. 
\item In Section \ref{sec:p-erg} we remind basics of the $p$-adic ergodic
theory in connection to automata functions.
\item Section \ref{sec:p-real} contains main results of the paper: By plotting
an automaton function into  real unit square we establish the automata 0-1 law
and find sufficient conditions for a system to be completely transitive or
absolutely transitive.
\item We conclude in Section \ref{sec:Concl}.
\end{itemize}  
\section{The $p$-adic representation of automata functions}
\label{sec:p-adic}
Every (left-)infinite word $\ldots\chi_2\chi_1\chi_0$ over the alphabet $\F_p$
can be associated to the $p$-adic integer $\chi_0+\chi_1p+\chi_2p^2+\cdots$
which is an element of the ring  $\Z_p$ of $p$-adic integers; the ring $\Z_p$
is a  complete  compact metric space w.r.t. $p$-adic metric (we remind the notion
below). The automaton $\mathfrak A(s_0)$ maps infinite words
to infinite words. Denote the corresponding  mapping via $f=f_{\mathfrak A(s_0)}$;
then $f$ is a function defined on  $\Z_p$ and valuated
in $\Z_p$. The function $f=f_{\mathfrak A(s_0)}\:\Z_p\>\Z_p$ is called the
\emph{automaton function} of the automaton $\mathfrak A(s_0)$. For instance,
the automaton function $f_{\mathfrak O(1)}$ of the 2-adic adding machine $\mathfrak O(1)$
is the \emph{2-adic odometer}, the transformation $f_{\mathfrak O(1)}(x)=x+1$ of the ring $\Z_2$ of 2-adic integers; whereas the automaton function $f_{\mathfrak C(0)}$
of the automaton
$\mathfrak C(0)$ from Example \ref{exm:const} is a constant function on $\Z_2$: $f_{\mathfrak C(0)}(x)=
\sum_{i=0}^\infty\alpha_i2^i\in\Z_2$.

Due to the fact that at every moment $n=0,1,2,\ldots$, the $n$-th output symbol
may depend only on the input symbols $\chi_0,\chi_1,\ldots,\chi_n$ that have been
feeded to the automaton at
the moments $0,1,\ldots,n$ respectively, the \emph{automaton function 
is a $p$-adic 1-Lipschitz function}; that is, $f$ satisfies the $p$-adic Lipschitz condition with the constant 1 w.r.t. $p$-adic metric and thus $f$
is a $p$-adic continuous function.
Vice versa, given a 1-Lipschitz function $f\:\Z_p\>\Z_p$, there exists an
automaton $\mathfrak A(s_0)$ such that $f=f_{\mathfrak A(s_0)}$, see further
Theorem
\ref{thm:auto-1L}. Therefore to
study the behavior of the system $\mathfrak  A$ we may (and will) study corresponding
automata functions rather than automata themselves; and
to study the behaviour of the latter functions 
we may apply techniques from
$p$-adic analysis and $p$-adic dynamics, see \cite{AnKhr}. This is the key point of our approach.

We remind that the space $\Z_p$ is the completion of the ring $\Z=\{0,\pm1,\pm
2,\ldots\}$ of (rational) integers w.r.t. the $p$-adic metric $d_p$ which is defined
as follows: given $a,b\in\Z$, if $a\ne b$ then denote $p^{\ord_p(a-b)}$  the largest
power of $p$ that divides $a-b$ and put $d_p(a,b)=\|a-b\|_p=p^{-\ord_p(a-b)}$, put $\|a-b\|_p=0$
if $a=b$. The $p$-adic metric violates the Archimedean Axiom and thus is called a non-Archimedean metric (or, an ultrametric). Now we describe our
approach less formally. 

Multiplication and addition of infinite words over alphabet $\F_p$ can be defined via
school-textbook-like
algorithms for multiplication/addition of integers represented by  base-$p$
expansions. For instance, in case of 2-adic integers (i.e., when $p=2$) the
following example shows that   $-1=\ldots 11111$ in $\Z_2$ (as
$\ldots0001=1$):

{\footnotesize
\begin{align*}
&\mbox{}&{}&\ldots 1&{}& 1&{}&1&{}&1&{}&\\
&\mbox{$+$}&{}&{}\\
&\mbox{}&{}& \ldots 0&{}&0&{}&0&{}&1&{}&\\
\intertext{\hbox to 3.5cm{}\hbox to 7cm{\hrulefill}}
&\mbox{}&{}&\ldots 0&{}&0&{}&0&{}&0&{}&
\end{align*} 
}

The next example shows that $\ldots1010101=-\frac{1}{3}$ in $\Z_2$ (as
$\ldots00011=3$):

{\footnotesize
\begin{align*}
&\mbox{}&{}&\ldots 0&{}&1&{}& 0&{}&1&{}&0&{}&1\\
&\mbox{$\times$}&{}&{}\\
&\mbox{}&{}& \ldots 0&{}&0&{}&0&{}&0&{}&1&{}&1\\
\intertext{\hbox to 3.2cm{}\hbox to 9.3cm{\hrulefill}}
&\mbox{}&{}&\ldots 0&{}&1&{}& 0&{}&1&{}&0&{}&1\\
&\mbox{$+$}&{}&{}\\
&\mbox{}&{}& \ldots 1&{}&0&{}&1&{}&0&{}&1&{}&\\
\intertext{\hbox to 3.2cm{}\hbox to 9.3cm{\hrulefill}}
&\mbox{}&{}&\ldots 1&{}&1&{}&1&{}&1&{}&1&{}&1
\end{align*}
}

The set of all infinite words over the alphabet $\F_p$ with so
defined operations (and distance) constitutes the ring (and the metric space) $\Z_p$. Note that $\Z_p$ contains the ring of all (rational) integers $\Z$
as well as some other elements from the field $\Q$ of rational numbers;
so $\Z_p\cap\Q\varsupsetneqq\Z$. For instance, in $\Z_2$ the
sequences that contain only finite number of 1-s correspond  to non-negative
rational integers represented by their base-2 expansions (e.g., $\ldots 00011=3$);
the sequences that contain only finite number of 0-s correspond  to negative
rational integers (e.g., $\ldots 111100=-4$);
the sequences that are (eventually) periodic correspond to   rational
numbers that can be represented by irreducible fractions with odd denominators
(e.g., $\ldots1010101=-\frac{1}{3}$); and
non-periodic sequences correspond to no rational numbers. It is also worth noticing that when $p=2$, the 2-adic integers
representing negative rational integers may be regarded as \emph{2's complements}
of the latter (cf. e.g. \cite{VHDL,Knuth}). In computer science, 2-adic representations
of rational integers are also known as \emph{Hensel codes}, cf. \cite{Knuth}, after the name of German mathematician Kurt
Hensel who discovered $p$-adic numbers more than a century ago.

By the definition, given two infinite words $\ldots\chi_2\chi_1\chi_0$
and $\ldots\xi_2\xi_1\xi_0$ over the alphabet $\F_p$,  the
distance $d_p$ between these words  is $p^{-n}$, where $n=\min\{i=0,1,2,\ldots
:
\chi_i\ne\xi_i\}$, and  the distance is 0 if no such $n$ exists. For instance, in the case  $p=2$ we have that
\begin{equation*}
\left.
\begin{aligned}
\ldots1010\underline{1}{0101}&={\footnotesize{-\frac{1}{3}}}\\
\ldots0000\underline{0}{0101}&=5\\
\end{aligned}
\right\}
\Rightarrow d_2\left(-\frac{1}{3},5\right)=\left\|\left(-\frac{1}{3}\right)-5\right\|_2=\frac{1}{2^4}=\frac{1}{16}.
\end{equation*}
In other words,
$-\frac{1}{3}\equiv 5\pmod{16}; -\frac{1}{3}\not\equiv 5\pmod{32}$. Note
that actually 
$\md p^k$, the \emph{reduction modulo $p^k$}, is an epimorphism of $\Z_p$ onto
the residue ring $\Z/p^k\Z$ modulo $p^k$ (we associate elements of the latter
ring
to $0,1,\ldots, p^k-1$):
\begin{equation}
\label{eq:md}
\md p^k\colon\Z_p\>\Z/p^k\Z;\ \left(\sum_{i=0}^\infty\alpha_ip^i\right)
\md p^k=\sum_{i=0}^{k-1}\alpha_ip^i,
\end{equation}
where $\alpha_i\in\F_p$. Thus, for $a,b\in\Z_p$, the following equivalences
hold:
\begin{equation}
\label{eq:md=ineq}
\|a-b\|_p\le p^{-k} \ \text{if and only if} \ a\md p^k=b\md p^k; \ \text{that
is, if and only if}\ a\equiv b\pmod{p^k}.
\end{equation}
Due to equivalence \eqref{eq:md=ineq}, one may use congruences between $p$-adic
numbers rather than inequalities for $p$-adic absolute values which is sometimes
more convenient during proofs. The advantage of using congruences rather
than inequalities in $p$-adic analysis over $\Z_p$ is that one may work with congruences by applying standard
number-theoretic techniques; e.g., add or multiply congruences sidewise,
etc. More about this in \cite{AnKhr}.

Metrics on  Cartesian powers $\Z_p^n$ can be  defined in a 
manner similar to that of the case $n=1$:
$$
\|(a_1,\ldots,a_n)-(b_1,\ldots,b_n)\|_p=\max\{\|a_j-b_j\|_p\colon j=1,2,\ldots,n\}
$$
for every $(a_1,\ldots,a_n),(b_1,\ldots,b_n)\in\Z_p^n$;
so $p$-adic continuous  multi-variate functions defined on and valuated
$\Z_p$ can be considered as well. 

Once the metric  is defined, one can speak of limits, of continuous functions,
of derivatives. of convergent series, etc.; that is, of $p$-adic Calculus.
We refer to the
numerous books on $p$-adic analysis (e.g., \cite{Gouvea:1997,Kat,Kobl,Sch} ) for further details.

An important example of continuous 2-adic functions are basic computer instructions,
both arithmetic (addition, multiplication, subtraction) and bitwise logical
($\AND$, the bitwise conjunction; $\OR$, the bitwise disjunction; $\XOR$,
the bitwise exclusive `or'; $\NOT$, the bitwise negation) and some others
(shifts towards higher order bits, masking). All these instructions can be
regarded as (univariate or two-variate) 1-Lipschitz functions defined on and valuated in the space of
2-adic integers $\Z_2$, \cite{AnKhr}. That is why the theory we develop finds numerous
applications
in computer science and cryptology: the straight-line
programs (and more complicated ones) combined from the mentioned instructions can also be regarded as continuous 2-adic mappings; so behaviour of these
programs can be analysed by techniques of the non-Archimedean dynamics,
see e.g. \cite{me-NATO,me-CJ,AnKhr,me:1,me:ex,me:2}. It is worth noticing here that similar approaches
work effectively also in genetics, cognitive sciences, image processing, quantum theory,
etc., see comprehensive monograph \cite{AnKhr} and references therein.

Concluding the section, we now give a formal proof that the class of all
automata functions $f_{\mathfrak A(s_0)}$ of automata of the form $\mathfrak A(s_0)=\langle\F_p,\Cal
S,\F_p,S,O,s_0\rangle$ coincides with the class of all 1-Lipschitz functions
$f\:\Z_p\>\Z_p$. The result is not new: It can be derived from a 
general result on asynchronous automata \cite[Theorem 2.4, Proposition 3.7]{Grigorch_auto};
in a special case $p=2$ the result was proved in  \cite{Vuillem_circ}. We use an opportunity to give a direct `$p$-adic' proof here  as we consider only synchronous
automata, and for arbitrary $p$.
\begin{thm}[Automata functions are 1-Lipschitz functions and vice versa]
\label{thm:auto-1L}
The automaton function $f_{\mathfrak A(s_0)}\:\Z_p\>\Z_p$ of the automaton
$\mathfrak A(s_0)=\langle\F_p,\Cal S,\F_p,S,O,s_0\rangle$ is 1-Lipschitz.

Conversely, for every 1-Lipschitz function $f\:\Z_p\>\Z_p$ there exists an
automaton 
 $\mathfrak A(s_0)=\langle\F_p,\Cal S,\F_p,S,O,s_0\rangle$ such that  $f=f_{\mathfrak A(s_0)}$.
\end{thm}
\begin{proof}
Given a $p$-adic integer $z\in\Z_p$, denote via $\delta_i(z)\in\F_p$, the
$i$-th `$p$-adic digit' of $z$; that is,  the $i$-th term coefficient
in the $p$-adic representation of 
$
z=\sum_{i=0}^\infty\delta_i(z)p^i.
$
As $s_i=S(\delta_{i-1}(z),s_{i-1})$ for every $i=1,2\ldots$,  the $i$-th output symbol
$\xi_i=\delta_i(f_{\mathfrak A}(z))$ depends only on input symbols
$\chi_0,\chi_1,\ldots,\chi_i$; that is 
\[
\delta_i(f_{\mathfrak A}(z))=\psi_i(\delta_0(z),\delta_1(z),\ldots,\delta_i(z))
\]
for all $i=0,1,2,\ldots$ and for suitable mappings $\psi_i\:\F_p^{i+1}\>\F_p$.
That is,
$f=f_{\mathfrak A(s_0)}\:\Z_p\>\Z_p$ is of the form
\begin{equation}
\label{eq:tri}
f\colon x=\sum_{i=0}^\infty\chi_ip^i\mapsto f(x)=\sum_{i=0}^\infty\psi_i(\chi_0,\ldots,\chi_i)p^i.
\end{equation}
This means that the function $f_{\mathfrak
A(s_0)}$ is 1-Lipschitz by \cite[Proposition 3.35]{AnKhr} as the mentioned proposition in application to the mappings we consider here can be re-stated as follows: A mapping $f\colon\Z_p\>\Z_p$ is 1-Lipschitz
if and only if  $f$ can be represented in the form \eqref{eq:tri}
for suitable mappings $\psi_i\:\F_p^{i+1}\>\F_p$, $i=0,1,2,\ldots$.

Conversely, let $f\:\Z_p\>\Z_p$ be a 1-Lipschitz mapping; then by \cite[Proposition 3.35]{AnKhr}  $f$ can be
represented in the form \eqref{eq:tri} for suitable mappings $\psi_i\:\F_p^{i+1}\>\F_p$, $i=0,1,2,\ldots$.
We now construct an automaton  $\mathfrak A(s_0)=\langle\F_p,\Cal S,\F_p,S,O,s_0\rangle$ such that $f_{\mathfrak A(s_0)}=f$. 

Let  $\F_p^\star$ be a set of all non-empty finite words over the alphabet $\F_p$.
We consider these words  as base-$p$ expansions of numbers from $\N=\{1,2,3,\ldots\}$
and enumerate all these words  by integers $1,2,3,\ldots$ in 
radix order in accordance with
the natural order on $\F_p$, $0<1<2<\cdots<p-1$: 
$$
0<1<2<\ldots<p-1<00<01<02<\ldots<0(p-1)<10<11<12<\ldots;
$$
so that $\nu(0)=1,\nu(1)=2,\nu(2)=3,\ldots,\nu(p-1)=p,\nu(00)=p+1,\nu(01)=p+2,\ldots$.
This way we establish a
one-to-one correspondence between the words $w\in\F_p^\star$ and integers $i\in\N$: $w \leftrightarrow \nu(w)$,  $i\leftrightarrow \omega(i)$ ($\nu(w)\in\N$, $\omega(i)\in\F_p^\star$).
Note that $\nu(\omega(i))=i$, $\omega(\nu(w))=w$ for all $i\in\N$ and all
non-empty words from
$w\in\F_p^\star$.
Define  $\omega(0)$ to be the empty word. 

Now put $\Cal S=\N_0=\{0,1,2,3,\ldots\}$, the set of all states of the automaton  $\mathfrak A(s_0)$ under construction, and take  the initial state  $s_0=0$. The state transition function $S$ is defined as follows: 
\begin{equation}
\label{eq:auto-st}
S(r,i)=\nu(r\circ\omega(i)),
\end{equation}
where $i=0,1,2,\ldots$ and  $r\in\F_p$. That is,
$S(r,i)$ is the number of the word  $r\circ\omega(i)$ which is a concatenation
of the word   $\omega(i)$  (the word whose number is $i$), the prefix, with the single-letter
word  $r$, the suffix. 

Now consider a one-to-one mapping 
$\theta_{n}(\chi_{n-1}\cdots\chi_1\chi_0)=(\chi_{0},\chi_1,\ldots,\chi_{n-1})$ from the $n$-letter words onto $\F_p^n$ and 
define the output function of the automaton $\mathfrak A(0)$ as follows:
\begin{equation}
\label{eq:auto-out}
O(r,i)=\psi_{\Lambda(\omega(i))}(\theta_{\Lambda(\omega(i))+1}(r\circ\omega(i))),
\end{equation}
where $i=0,1,2,\ldots$ and  $r\in\F_p$. Remind that we denote via $\Lambda(w)$
the length of the word $w$. 

The idea of the construction is illustrated by Figure \ref{fig:constr} which depicts Moore diagram of the  automaton
$\mathfrak A(0)$ for the case $p=2$:
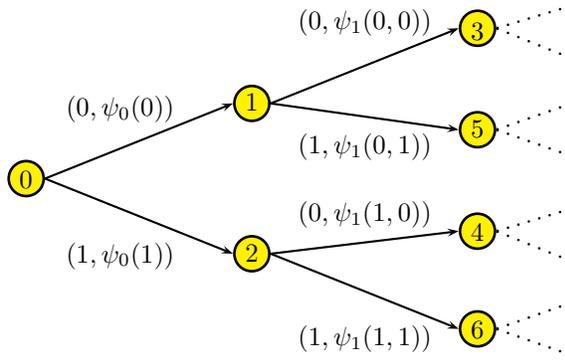
\begin{figure}[h]
\begin{quote}\psset{unit=0.5cm}
 \begin{pspicture}(-1.2,-2)(24,10)
\pscircle[fillstyle=solid,fillcolor=yellow,linewidth=1pt](4,4){.5}
\pscircle[fillstyle=solid,fillcolor=yellow,linewidth=1pt](10,6){.5}
\pscircle[fillstyle=solid,fillcolor=yellow,linewidth=1pt](10,2){.5}
\pscircle[fillstyle=solid,fillcolor=yellow,linewidth=1pt](16,8){.5}
\pscircle[fillstyle=solid,fillcolor=yellow,linewidth=1pt](16,5.3){.5}
\pscircle[fillstyle=solid,fillcolor=yellow,linewidth=1pt](16,2.6){.5}
\pscircle[fillstyle=solid,fillcolor=yellow,linewidth=1pt](16,0){.5}
  \psline{->}(4.5,4)(9.5,6)
  \psline{->}(4.5,4)(9.5,2)
  \psline{->}(10.5,6)(15.5,8)
  \psline{->}(10.5,6)(15.5,5.3)
  \psline{->}(10.5,2)(15.5,2.6)
  \psline{->}(10.5,2)(15.5,0)
  \psline[linestyle=dotted,linewidth=1pt](16.5,0)(18.5,-.7)
  \psline[linestyle=dotted,linewidth=1pt](16.5,0)(18.5,.7)
  \psline[linestyle=dotted,linewidth=1pt](16.5,2.6)(18.5,1.8)
  \psline[linestyle=dotted,linewidth=1pt](16.5,2.6)(18.5,3.2)
  \psline[linestyle=dotted,linewidth=1pt](16.5,5.3)(18.5,4.6)
  \psline[linestyle=dotted,linewidth=1pt](16.5,5.3)(18.5,6)
  \psline[linestyle=dotted,linewidth=1pt](16.5,8)(18.5,7.2)
  \psline[linestyle=dotted,linewidth=1pt](16.5,8)(18.5,8.6)
\uput{1}[90](4,2.75){$0$}
\uput{1}[90](10,4.8){$1$}
\uput{1}[90](10,.8){$2$}
\uput{1}[90](16,6.7){$3$}
\uput{1}[90](16,4.1){$5$}
\uput{1}[90](16,1.35){$4$}
\uput{1}[90](16,-1.25){$6$}
\uput{1}[90](6.5,4.5){$(0,\psi_0(0))$}
\uput{1}[90](6.5,.6){$(1,\psi_0(1))$}
\uput{1}[90](13,6.8){$(0,\psi_1(0,0))$}
\uput{1}[90](13,3.5){$(1,\psi_1(0,1))$}
\uput{1}[90](13,1.7){$(0,\psi_1(1,0))$}
\uput{1}[90](13,-1.6){$(1,\psi_1(1,1))$}
\end{pspicture}
\end{quote}
\caption{Moore diagram of the automaton $\mathfrak A(0)$, $p=2$; so
$\omega(0)$ is the empty word, $\omega(1)=0$, $\omega(2)=1$, $\omega(3)=00$,
$\omega(4)=01$, $\omega(5)=10$, $\omega(6)=11$,\ldots}
\label{fig:constr}
\end{figure}

Now, as both $f$ and $f_{\mathfrak A(s_0)}$ are 1-Lipschitz, thus continuous with respect to the $p$-adic metric, and as $\N_0$ is dense in $\Z_p$,
to prove that
 $f=f_{\mathfrak A(s_0)}$ is suffices to show that
 \begin{equation}
 \label{eq:auto-1}
 f_{\mathfrak
A(s_0)}(\tilde w)\equiv f(\tilde w)\pmod{p^{\Lambda(w))}}
\end{equation}
for all finite non-empty words $w\in\F_p^\star$, where $\tilde w\in\N_0$ stands for the integer whose base-$p$
expansion is $w$.  We prove that \eqref{eq:auto-1} holds for all $w\in\F_p^\star$
once  $\Lambda(w)=n>0$ by induction on $n$.

If
$n=1$ then $\tilde w\in\F_p$; so once $w$ is feeded to
 $\mathfrak A$, the automaton reaches the state $S(w,0)=\nu(w)$ (cf. \eqref{eq:auto-st})
 and outputs
 $O(w,0)=\psi_0(\theta_1(w))
 \equiv f(\tilde w)\pmod p$ 
 (cf. \eqref{eq:auto-out}), see \eqref{eq:tri}. 
Thus, \eqref{eq:auto-1}
 holds in this case.

Now assume that \eqref{eq:auto-1} holds for all $w\in\F_p^\star$ such that $\Lambda(w)=n<k$ and prove that \eqref{eq:auto-1} holds also when $\Lambda(w)=n=k$.
Represent $w=r\circ v$, where $r\in\F_p$ and $\Lambda(v)=n-1$. By the induction
hypothesis, after the word   $v$ has been feeded to $\mathfrak A$, the automaton
reaches the state $\nu(v)$ and outputs the word $v_1$ of length $n-1$
such that
$\tilde v_1\equiv f(\tilde v)\md p^{n-1}$. Next, being feeded by the letter $r$, the
automaton (which is in the state $\nu(v)$ now) outputs the letter
$O(r,\nu(v))=\psi_{\Lambda(\omega(\nu(v)))}(\theta_{\Lambda(\omega(\nu(v)))+1}(r\circ\omega(\nu(v))))=
\psi_{\Lambda(v)}(\theta_{\Lambda(v)+1}(r\circ v))$.
This means that once feeded by $w$, the automaton $\mathfrak A(s_0)$ 
outputs the word
$v_2=(\psi_{\Lambda(v)}(\theta_{\Lambda(v)+1}(r\circ v)))\circ v_1$. Now
note that $\tilde v_2\equiv
f(\tilde
w)\pmod{p^n}$ by \eqref{eq:tri}.
\end{proof}
\begin{note}
From the proof of Theorem \ref{thm:auto-1L} it is clear that the mapping
 $f_{n,\mathfrak
A(s_0)}\:\Z/p^n\Z\>\Z/p^n\Z$ is just a reduction  modulo $p^n$ of the automaton function $f_{\mathfrak
A(s_0)}$: $f_{n,\mathfrak
A(s_0)}=f_{\mathfrak
A(s_0)}\md p^n$ for all $n=1,2,3,\ldots$.
\end{note}
\begin{note}
In automata theory, \emph{word transducers} (or, \emph{asynchronous
automata}) are also considered; the
latter are automata that allow (possibly empty) words as output for each transition. Although the automata we  consider  are all synchronous (i.e..,
letter transducers  rather than word transducers),
it is worth mentioning here that the automaton function of a word transducer
whose input/output alphabets are $\F_p$ can also be considered as a continuous
(however, not necessarily 1-Lipschitz any longer) mapping
from $\Z_p$ to $\Z_p$ once the transducer is non-degenerate, see
\cite[Theorem 2.4]{Grigorch_auto}. 
\end{note}
Further in the paper, given a 1-Lipschitz function $f\:\Z_p\>\Z_p$, via $\mathfrak A_f(s_0)$
we denote an automaton $\langle\F_p,\Cal S,\F_p,S,O,s_0\rangle$ whose automaton function is $f$; that is, $f_{\mathfrak A_f(s_0)}=f$. Note that given $f$, the automaton $\mathfrak A_f(s_0)$ is \emph{not} unique: There are numerous automata
that have the same automaton function $f$. Nonetheless, this non-uniqueness will not
cause misunderstanding since in the paper we are mostly interested with automata
functions rather than with `internal structure' (e.g., with state sets, state
transition and output functions, etc.) of automata themselves.

\section{The $p$-adic ergodic theory and transitivity of automata}
\label{sec:p-erg}
The ring $\Z_p$ can be endowed with a probability measure $\mu_p$:
Elementary $\mu_p$-measurable sets are balls $B_{p^{-r}}(a)=a+p^r\Z_p=\{z\in\Z_p\colon z\equiv a\pmod{p^r}\}$  of radii $p^{-r}$, $r=1,2,\ldots$, centered
at $a\in\Z_p$. In other words, the ball $B_{p^{-r}}(a)$ is a  set of all infinite words over alphabet $\F_p=\{0,1,\ldots,p-1\}$
that have common prefix of length  $r$. We put $\mu_p(B_{p^{-r}}(a))=p^{-r}$
thus endowing $\Z_p$ with a probability measure $\mu_p$ (which actually is a normalized
Haar measure). Note that all 1-Lipschitz mappings $f\:\Z_p\>\Z_p$ are $\mu_p$-measurable
(i.e., $f^{-1}(S)$ is $\mu_p$-measurable once $S\subset\Z_p$ is $\mu_p$-measurable).

A $\mu_p$-measurable mapping $f\colon \Z_p\>\Z_p$ is called  \emph{ergodic} if  the two following
conditions hold simultaneously: 
\begin{enumerate}
\item $f$ \emph{preserves the measure} $\mu_p$; i.e., $\mu_p(f^{-1}(S))=\mu_p(S)$ for each $\mu_p$-measurable subset 
$S\subset \Z_p$, and   
\item $f$ has no proper invariant $\mu_p$-measurable 
subsets:
$f^{-1}(S)=
S$ 
{implies either} 
$\mu_p(S)=0$, 
{or} 
$\mu_p(S)=1$. 
\end{enumerate}
A family $\Cal F=\{f_i\:i\in I\}$ of $\mu_p$-measurable mappings $f_i\:\mathbb Z_p\>\mathbb
Z_p$ (which are not necessarily measure-preserving mappings) 
is called \emph{ergodic} if the mappings $f_i$, $i\in I$, have no common $\mu_p$-measurable invariant subset 
other than sets of measure 0 or 1; that is, if there exists a $\mu_p$-measurable subset $S\subset\mathbb Z_p$ such that
$f^{-1}_i(S)= S$ for all $i\in I$, then necessarily either $\mu_p(S)=0$, or
$\mu_p(S)=1$.

Note that in the paper speaking of  ergodicity of a \emph{single}  mapping we always mean
the mapping is measure-preserving; whereas
in general ergodic theory the non-measure-preserving ergodic mappings (the ones
that satisfy only the second condition (ii) of the above two) are sometimes
also concerned.
To illustrate the notion of ergodicity we use, consider a finite set $M$ endowed with a natural probability
measure $\mu(A)=\#A/\#M$ for all $A\subset M$ (where $\#A$ stands for the
number of elements in $A$). The measure-preservation of
the mapping $f\:M\>M$ is equivalent to the bijectivity of $f$, whereas the
ergodicity of $f$ (when respective conditions (i) and (ii) hold simultaneously) is equivalent to the transitivity of the mapping $f$ in the meaning of Definition \ref{def:trans-1}; and the ergodicity of the family $\Cal F$ of mappings $f_i\:M\> M$, $i\in I$, is equivalent to the transitivity of the family
$\Cal F$ in the meaning of Definition \ref{def:trans}.

As in the paper we deal with the only measure $\mu_p$, so further speaking
of measure-preservation (as well as of measurability and of  ergodicity) we omit mentioning the respective measure.
From the $p$-adic ergodic theory (see \cite{AnKhr}) the following theorem
can be deduced: 
\begin{thm}
\label{thm:word-trans=erg}
A system $\mathfrak A=\langle\F_p,\Cal S,\F_p,S,O\rangle$ is
word transitive 
if and only if the automaton function $f_{\mathfrak
A(s_0)}$ on $\Z_p$ is ergodic.  If the system $\mathfrak A$ is completely
transitive, the
family $f_{\mathfrak A(s)}$, $s\in\Cal
S$, of automata functions is ergodic.
\end{thm}
Remind that under conventions from
the beginning of the paper, $s_0$ is the state of
the system $\mathfrak A$ such that all other states are reachable from $s_0$.

Theorem \ref{thm:word-trans=erg} implies a number of various methods to determine the word transitivity of
automata: For instance, a \emph{binary} automaton $\mathfrak P$ (that is
an automaton with
a binary input and output; i.e., with
$p=2$) whose automaton function $f_{\mathfrak P}$ is a polynomial with integer coefficients (i.e., $f_\mathfrak P=g$ where $g(x)\in\Z[x]$) is word transitive if and only if it is 3-word transitive; that is, the transitivity of $\mathfrak
P$ on the set $W_3$
of all binary words of length 3 is equivalent to the transitivity of $\mathfrak
P$ on the
set $W_n$ of all binary words of length $n$, for all $n=1,2,3,\ldots$. Moreover,
a
binary automaton $\mathfrak F$ is word transitive if and
only if its automaton function is of the form $f_{\mathfrak
F}(x)=1+x+2(g(x+1)-g(x))$, where $g=g_{\mathfrak G}$ is an automaton function
of some binary automaton $\mathfrak G$. For other results of this sort and
for the whole $p$-adic ergodic theory see \cite{AnKhr}. Although complete
transitivity of the system $\mathfrak A=\langle\F_p,\Cal S,\F_p,S,O\rangle$ is also related to ergodicity; however, to the ergodicity of the family of automata functions $f_{\mathfrak
A(s)}$, $s\in\Cal S$, cf. Definition \ref{def:auto-trans} and Theorem \ref{thm:word-trans=erg}, rather than to
the ergodicity of a single automaton function $f_{\mathfrak A(s_0)}$. This is why to determine complete/absolute transitivity rather than just word
transitivity we need some more sophisticated techniques that are discussed
further.

\section{Plots of automata functions on the real plane}
\label{sec:p-real}
Remind that under conventions from
the beginning of the paper, there exists a state $s_0$ of
the system $\mathfrak A$ such that all other states are reachable from $s_0$;
so although further results of the paper are stated mostly for automata, they hold
for systems as well.

Given an automaton $\mathfrak A(s_0)$, 
consider the corresponding automaton function $f=f_{\mathfrak A(s_0)}\colon\Z_p\rightarrow\Z_p$. Denote $E_k(f)$ the set of  all the following points $e_k^f(x)$  of closed Euclidean unit square $\mathbb I^2=[0;1]\times[0;1]\subset\R^2$: 
$$e_k^f(x)=\left(\frac{x\md p^k}{p^k},\frac{f(x)\md p^k}{p^k}
\right),$$
where $x\in\Z_p$ and $\md p^k$ is a reduction 
modulo $p^k$, cf. \eqref{eq:md}.
Note that $x\md p^k$ corresponds to the 
prefix of length $k$ of the infinite word $x\in\Z_p$, i.e., to the input word of length $k$  of the automaton
$\mathfrak A(s_0)$; while
$f(x)\md p^k$ corresponds to the respective output  word of length $k$.
That is, given an input word $w=\chi_{k-1}\cdots\chi_1\chi_0$
and the corresponding output word  $w^\prime=\xi_{k-1}\cdots\xi_1\xi_0$,
we consider in $\mathbb I^2$ the set of all points 
$$(\chi_{k-1}p^{-1}+\cdots+\chi_1p^{-k+1}+\chi_0p^{-k},
\xi_{k-1}p^{-1}+\cdots+\xi_1p^{-k+1}+\xi_0p^{-k}),$$
for all pairs $(w,w^\prime)$ of input/output words of length $k$.
\begin{figure}[h]
\begin{quote}\psset{unit=0.5cm}
 \begin{pspicture}(0,8)(24,12)

  \psline[linecolor=black]{<-}(20,11)(14,11)
  \psline[linecolor=black]{<-}(10,11)(3.7,11)

  \psline[linecolor=black](3.7,11.3)(3.7,10.7)

  \psframe[fillstyle=solid,fillcolor=yellow,linecolor=black,linewidth=2pt](10,10)(14,12)
  
  \uput{0}[90](12,10.7){$\mathfrak A$}
  
  \uput{0}[90](1.9,9.2){$x\md p^k=$}
  \uput{0}[90](22.6,9.2){$=f(x)\md p^k$}
  \uput{0}[90](11.9,8.7){${\underbrace{{\black
  \chi_{k-1}\cdots\cdots\cdots\chi_1\chi_0\ \ \ \ \ \ \ \ \ \ \ \ \ \ \ \ \ \ \ \ \ \ \ \ \ \ \ \xi_{k-1}\cdots\cdots\cdots\xi_1\xi_0}}}
  $}
  
  \uput{0}[90](12,7.5){{$e_k^f(x)=(0.\chi_{k-1}\ldots\chi_1\chi_0,0.\xi_{k-1}\ldots\xi_1\xi_0)$}}
  
 \end{pspicture}
\end{quote}

\end{figure}

 The set $E_k(f)$ may be considered as a sort of a plot 
of the automaton function $f$ on  the real plane $\R^2$.  The plot characterizes behaviour
of the automaton; namely, it can be observed that basically the behaviour is of  two types only:  
\begin{enumerate}
\item as $k\to\infty$, the point set $E_k(f)$ 
is getting more and more dense 
(cf. Fig. 
\ref{fig:Quad-16}--\ref{fig:Quad-23}, $p=2$) , or 
\item  $E_k(f)$ is getting less and less dense while $k\to\infty$,
cf. Fig. 
\ref{fig:KlSh-lac-16}--\ref{fig:KlSh-lac-22} ($p=2$).
\end{enumerate}
\begin{figure}[h]
\begin{minipage}[b]{.495\linewidth}
\epsfig{file=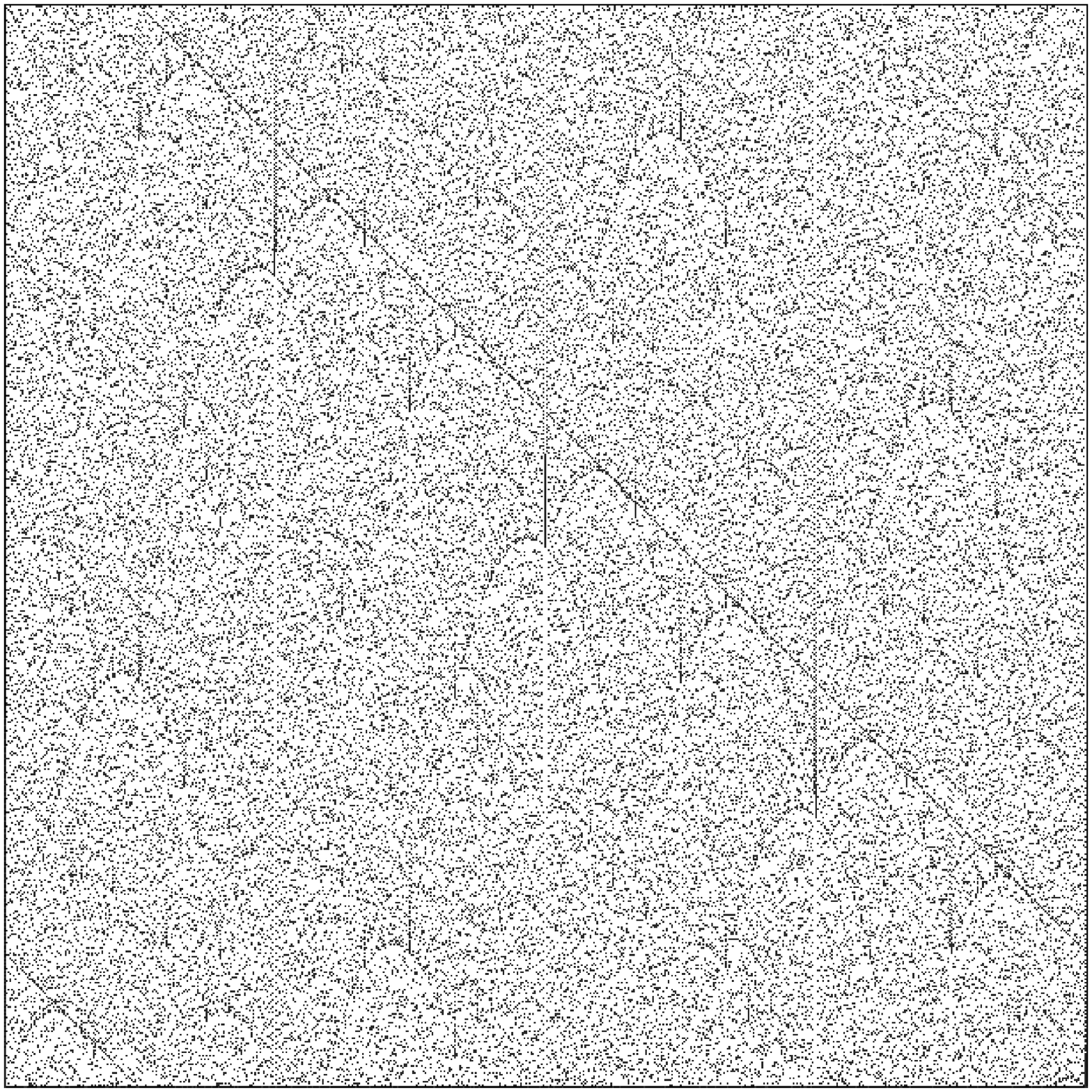,width=\linewidth}
\caption{%
{$f(x)=2x^2+3x+1$, $k=16$}}
\label{fig:Quad-16}
\end{minipage}\hfill
\begin{minipage}[b]{.495\linewidth}
\epsfig{file=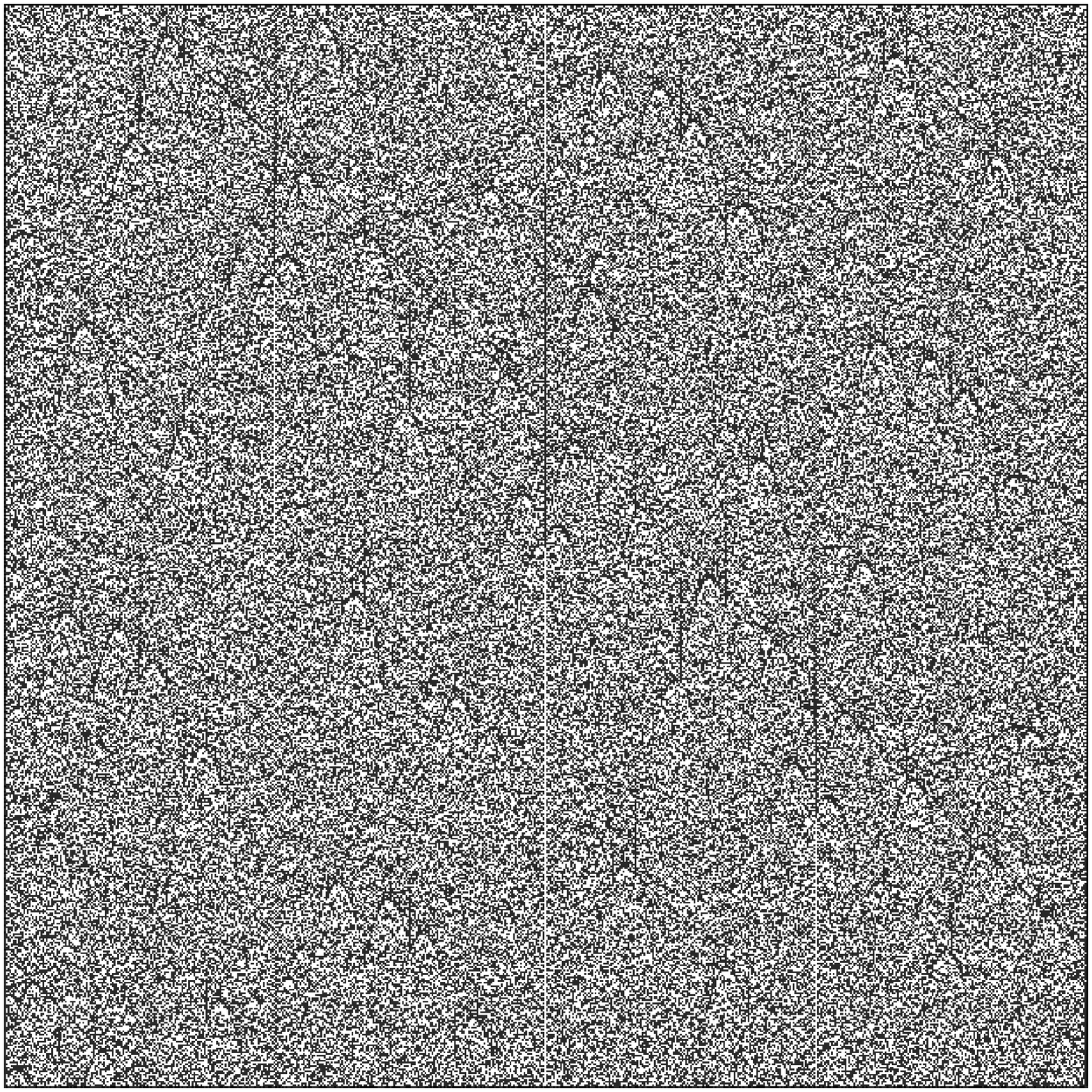,width=\linewidth}
\caption{%
{Same function, $k=18$}}
\label{fig:Quad-18}
\end{minipage}
\begin{minipage}[b]{.495\linewidth}
\epsfig{file=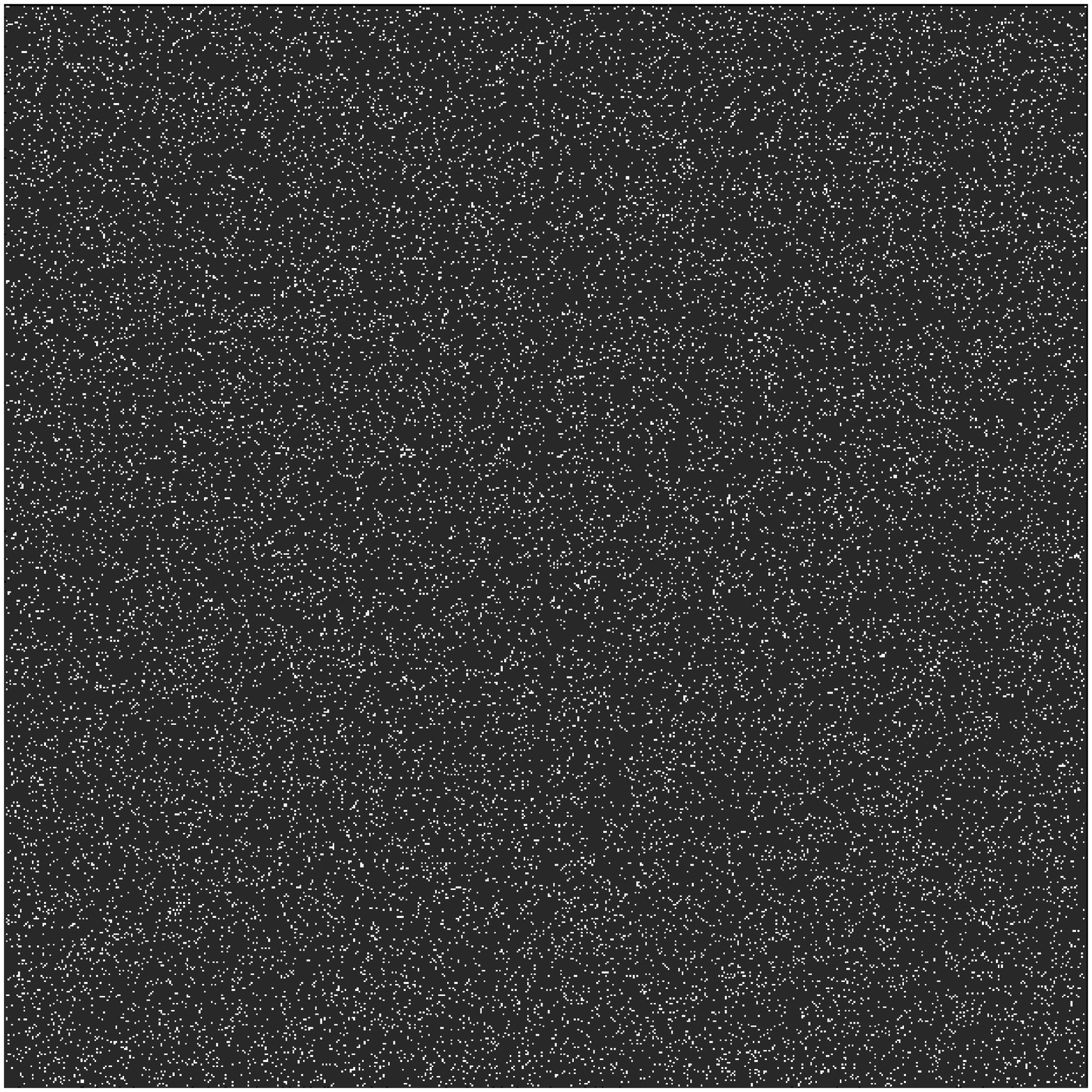,width=\linewidth}
\caption{%
{Same function, $k=20$}}
\label{fig:Quad-20}
\end{minipage}\hfill
\begin{minipage}[b]{.495\linewidth}
\epsfig{file=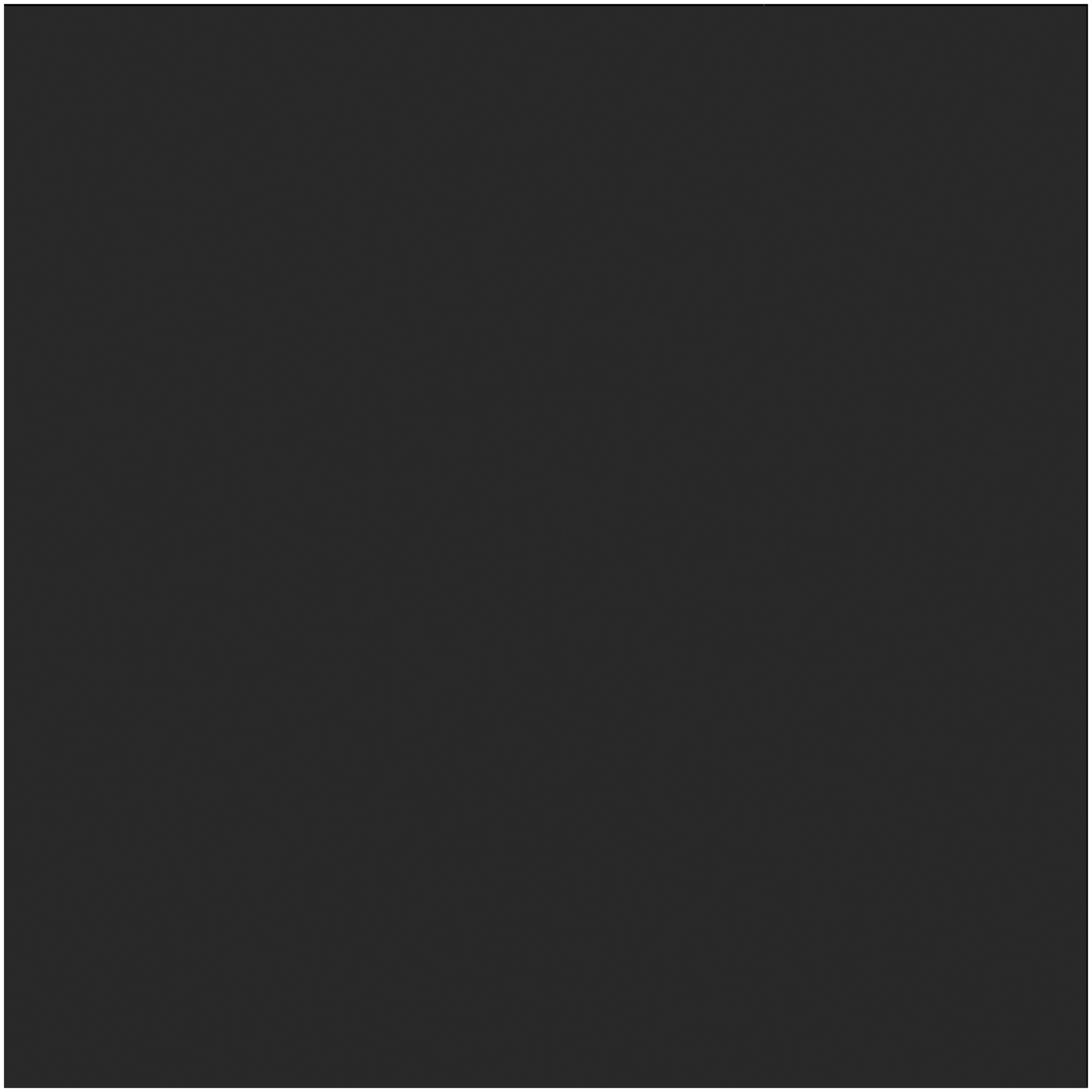,width=\linewidth}
\caption{%
{Same function, $k=23$}}
\label{fig:Quad-23}
\end{minipage}
\end{figure}

\begin{figure}[th]
\begin{minipage}[b]{.495\linewidth}
\epsfig{file=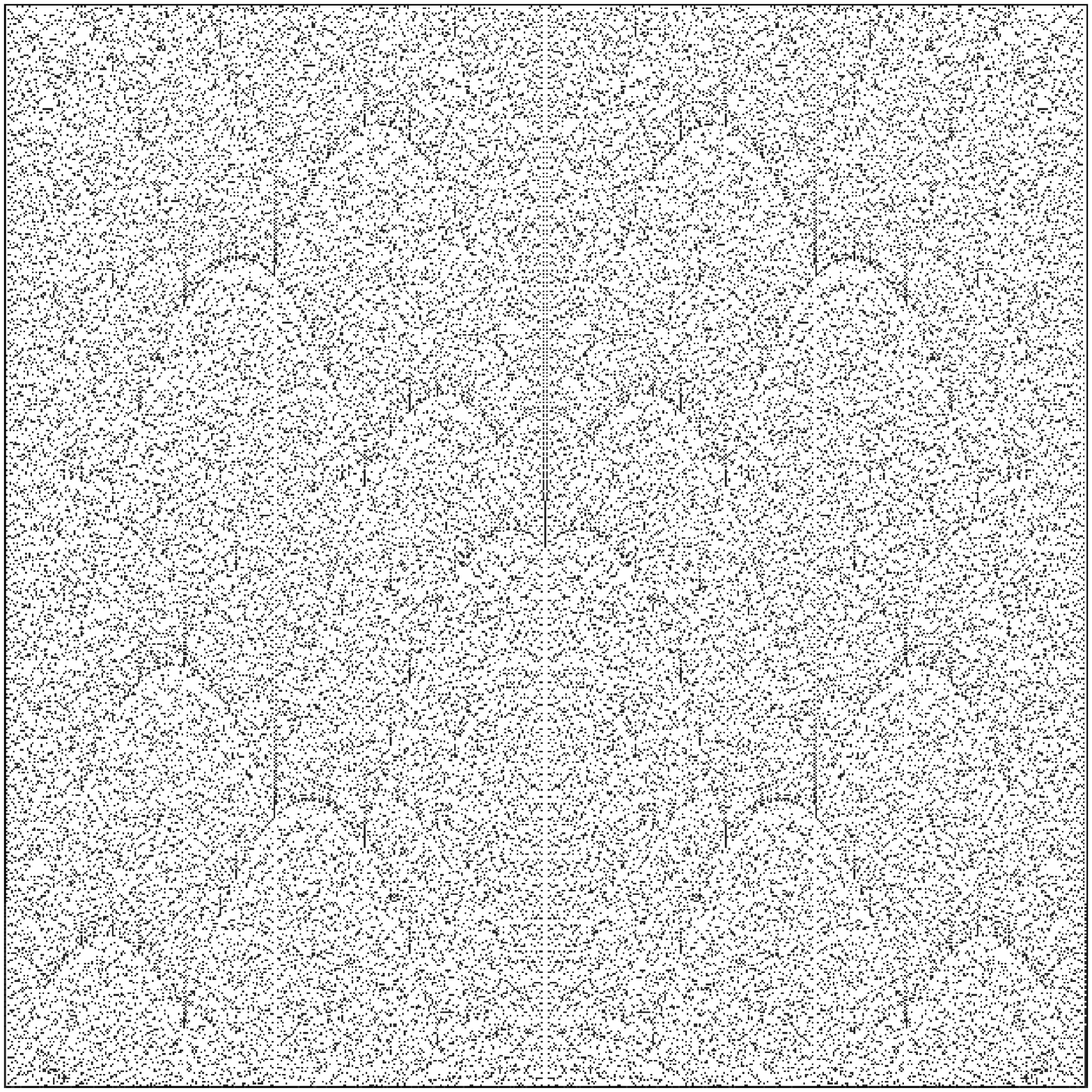,width=\linewidth}
\caption{%
{$f(x)=
x+x^2\OR 
(-131065)$,
$k=16$}}
\label{fig:KlSh-lac-16}
\end{minipage}\hfill
\begin{minipage}[b]{.495\linewidth}
\epsfig{file=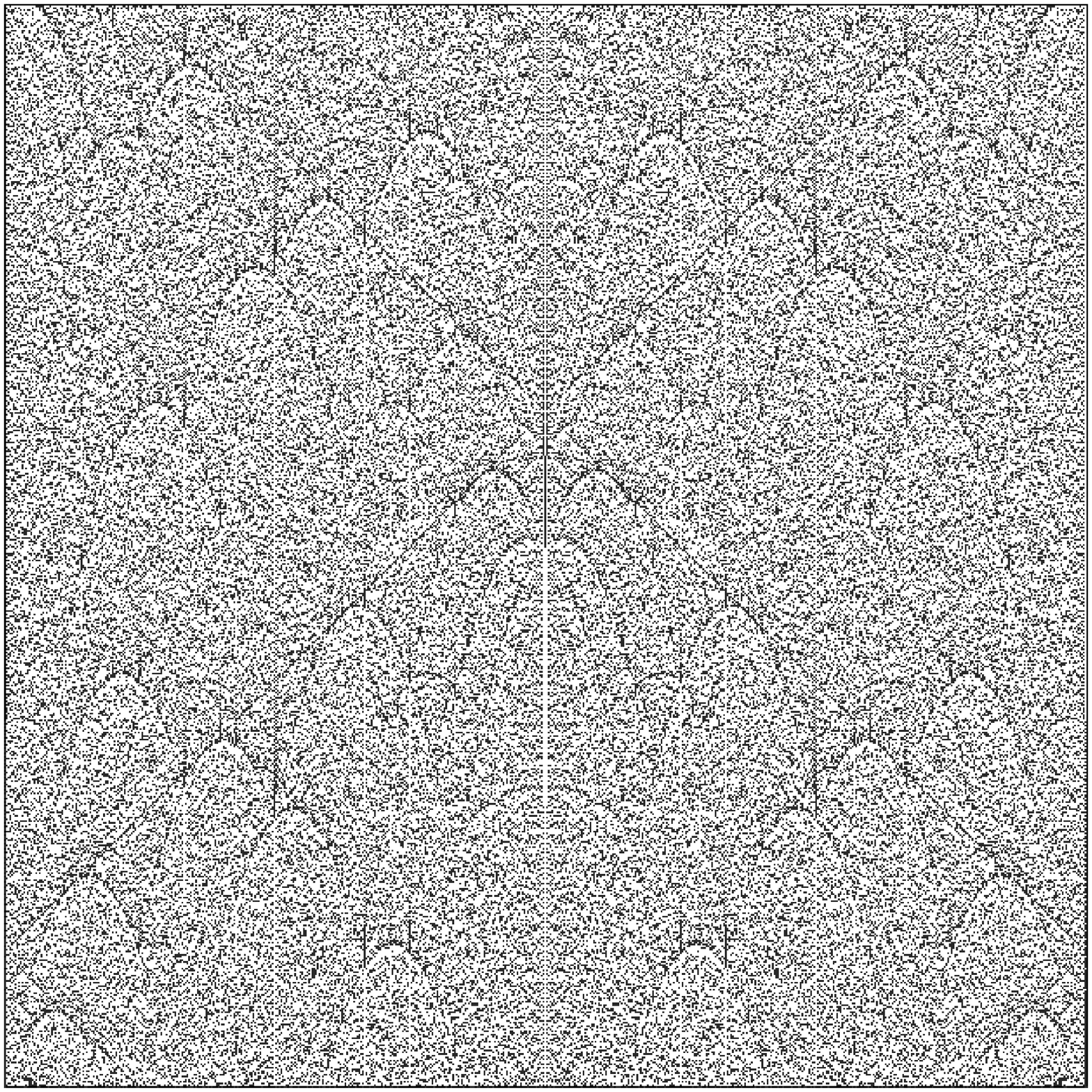,width=\linewidth}
\caption{%
{Same function, $k=17$}}
\label{fig:KlSh-lac-17}
\end{minipage}
\begin{minipage}[b]{.495\linewidth}
\epsfig{file=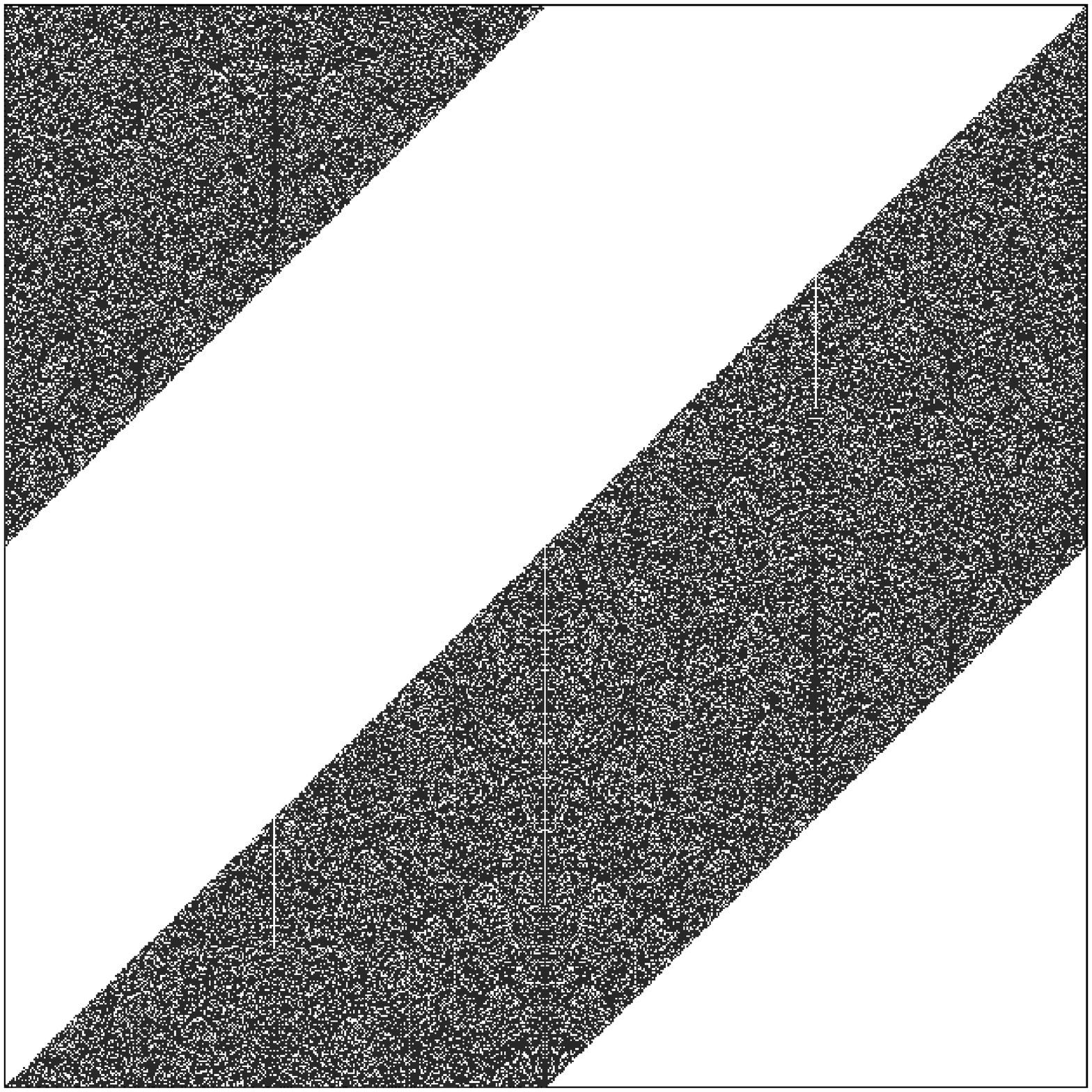,width=\linewidth}
\caption{%
{Same function, $k=18$}}
\label{fig:KlSh-lac-18}
\end{minipage}\hfill
\begin{minipage}[b]{.495\linewidth}
\epsfig{file=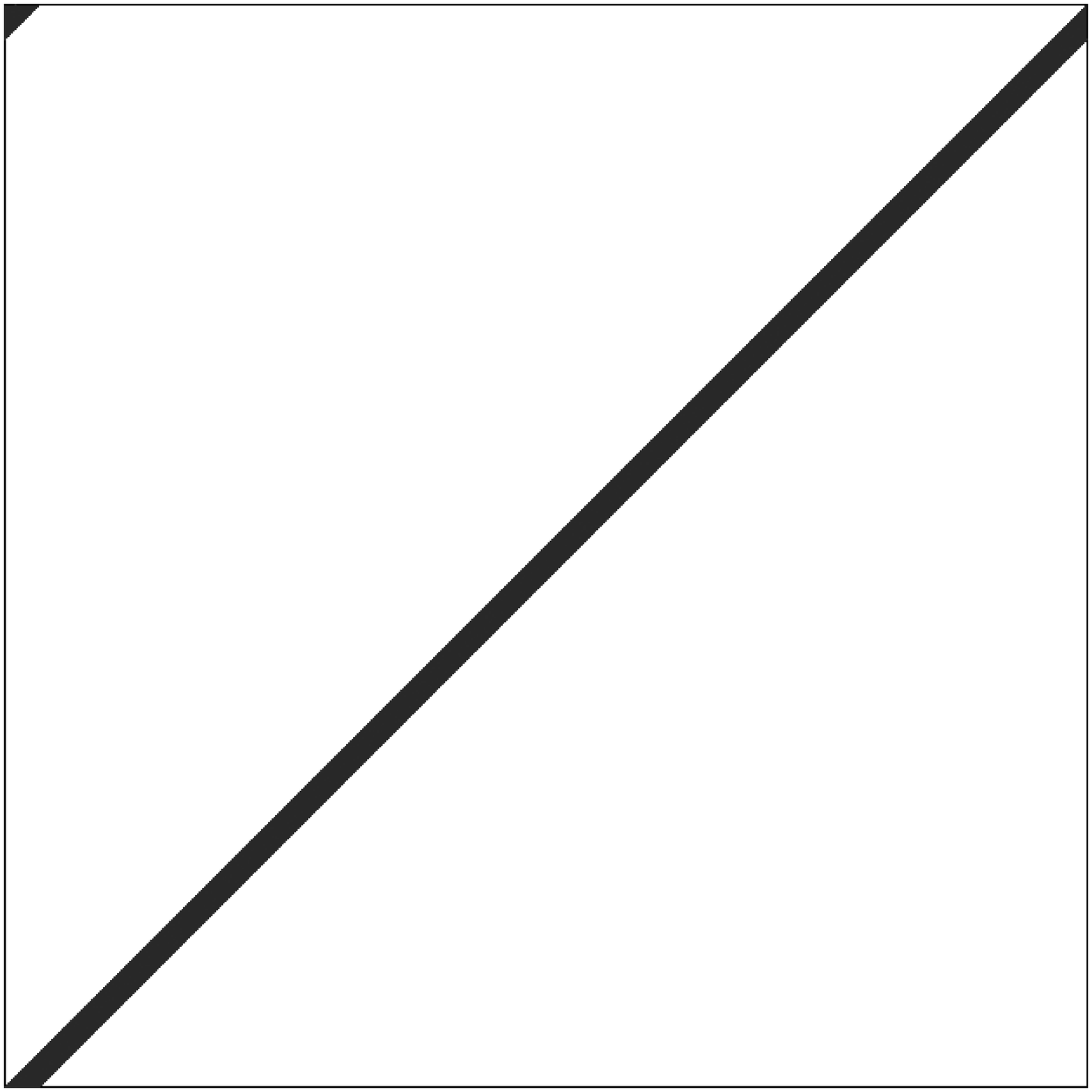,width=\linewidth}
\caption{%
{Same function, $k=22$}}
\label{fig:KlSh-lac-22}
\end{minipage}
\end{figure}
It is intuitively clear that, say, for pseudorandom number generation automata
of type (i) are preferable\footnote{For a deeper mathematical reasoning
see \cite{AnKhr}.}; so we need to explain/prove the phenomenon and to develop techniques in order to determine/construct automata of type
(i).
\subsection{The automata 0-1 law}
Denote $\Cal E(f)$ the closure of the set $E(f)=\bigcup_{k=1}^\infty E_k(f)$ in the topology of the real plane $\R^2$. As $\Cal E$ is closed, it is measurable
with respect to the Lebesgue measure on the real plane $\R^2$. Let  $\alpha(f)$ be the Lebesgue measure  
 of  $\Cal E(f)$. It is clear that $0\le\alpha(f)\le1$; but it turns out that in fact only two extreme cases occur: $\alpha(f)=0$ or $\alpha(f)=1$.
This is the first of the main results of the paper: 
\begin{thm}[The automata 0-1 law]
\label{thm:Auto_0-1}
For
$f$, the following alternative holds: 
Either 
$\alpha(f)=0$ \textup{(equivalently, $\Cal E(f)$ is nowhere dense in $\mathbb I^2$)},
or 
$\alpha(f)=1$ \textup{(equivalently, $\Cal E(f)=\mathbb I^2$)}.
\end{thm}

We note that  although Theorem \ref{thm:Auto_0-1} has been already announced,  
see \cite[Proposition 11.15]{AnKhr}, actually in \cite{AnKhr} only 
part of the statement  is proved (the one that concerns density of $\Cal E(f)$) whereas
the part that concerns the value of the Lebesgue measure is not. Remind 
that nowhere
dense sets can nevertheless have  positive Lebesgue measures, cf. fat Cantor sets (e.g. the Smith-Volterra-Cantor set), also known as \emph{$\epsilon$-Cantor
sets}, see e.g. \cite{RealAnalys}. Nonetheless, Theorem \ref{thm:Auto_0-1} is true; a complete proof follows.
\begin{proof}[Proof of Theorem \ref{thm:Auto_0-1}]
Let $\alpha(f)>0$; we are going to prove that then $\alpha(f)=1$ and $\Cal
E(f)=\mathbb I^2$. 

Either of the two following cases is possible: 1) Some point from $\Cal E(f)$ have
an open neighbourhood (in the unit square $\mathbb I^2$) that lies completely in $\Cal E(f)$, or, on the contrary, 2) no such point in $\Cal E(f)$ exists
(thus, $\Cal E(f)$ is nowhere dense in $\mathbb I^2$ then).
We consider the two cases separately and prove that within the first one necessarily
$\alpha(f)=1$ while the second one is impossible (that is, if $\Cal E(f)$ is nowhere dense in $\mathbb I^2$ then necessarily $\alpha(f)=0$). Given
$a,b\in\R$, $a\le b$, during the proof we denote via $(a;b)$ (respectively,
via $[a;b]$) the corresponding open interval (respectively, closed segment)
of the real line $\R$; while for $c,d\in\R$ we denote via $(c,d)$ the corresponding
point on the  real plane $\R^2$.

\textbf{Case 1}: In this case,  
there exist $u,v,u',v'$, $0\le u<v\le 1$, $0\le u'<v'\le 1$ such that
the closed square $[u;v]\times [u';v']\subset\mathbb I^2$ lies completely in $\Cal E(f)$,
and every point from the open real interval $(u';v')$ is a limit (with respect
to the standard Archimedean metric in $\R$)
of some sequence of fractions 
$u'<\frac{f(a_m)\md{p^m}}{p^{m}}<v'$, where $u<\frac{a_m}{p^{m}}<v$, $m=1,2,\ldots$.
Thus, we can take $n\in\N$ and 
$w=\omega_0+\omega_1\cdot p+\cdots+\omega_{n-1}\cdot p^{n-1}$, 
where $\omega_i\in\{0,1,\ldots,p-1\}$, $i=0,1,\ldots,n-1$, so that
the square  
$$
S=\left[\frac{w}{p^{n}};\frac{w}{p^{n}}+\frac{1}{p^{n}}\right]\times
\left[\frac{f(w)\md{p^n}}{p^{n}};\frac{f(w)\md{p^n}}{p^{n}}+\frac{1}{p^{n}}\right]
$$
lies completely in $\Cal E(f)$, and every inner point $(x,y)$ of the square
$S$ \footnote{that is, $(x,y)$ has an open neighborhood that lies completely
in $S$}
is a limit as $j\to\infty$ (with respect to the standard Archimedean metric in $\R^2$) of
a sequence of inner points 
$$
(r_j,t_j)=\left(\frac{z_j+p^{N_j}\cdot w}{p^{N_j+n}},\frac{f(z_j+p^{N_j}\cdot w)\md{p^{N_j+n}}}{p^{N_j+n}}\right)\in S,
$$
where $N_j\in\N$, $z_j\in\{0,1,\ldots,p^{N_j}-1\}$.

Now, as $f$ is a 1-Lipschitz mapping
from $\Z_p$ to $\Z_p$, for every $z\in\{0,1,\ldots,p^N-1\}$ we have that
$f(z+p^N\cdot w)\equiv (f(z)\md{p^N})+p^N\cdot\xi_N(z)\pmod{p^{N+n}}$ for
a suitable $\xi_N(z)\in\{0,1,\ldots,p^n-1\}$; thus,
$$
\frac{f(z+p^N\cdot w)\md{p^{N+n}}}{p^{N+n}}=\frac{f(z)\md{p^N}}{p^{N+n}}+
\frac{\xi_N(z)}{p^n}.
$$
Hence, $\xi_{N_j}(z_j)=f(w)\md{p^n}$ for all $j=1,2,\ldots$ as all $(r_j,t_j)$ are
inner points of $S$. Therefore, every inner point $(x,y)\in S$, which can be represented as
$$
(x,y)=\left(\frac{w}{p^n}+\frac{\chi}{p^n},
\frac{f(w)\md{p^n}}{p^n}+\frac{\gamma}{p^n}\right),
$$
where $\chi$ and $\gamma$ are real numbers, $0<\chi<1$, $0<\gamma<1$, is a limit (as $j\to\infty)$
of the point sequence 
$$
(r_j,t_j)=\left(\frac{w}{p^n}+\frac{z_j}{p^{N_j}}\cdot\frac{1}{p^n},
\frac{f(w)\md{p^n}}{p^n}+\frac{f(z_j)\md{p^{N_j}}}{p^{N_j}}\cdot\frac{1}{p^n}\right)\in S.
$$
From here it follows that every inner point $(\chi,\gamma)\in\mathbb I^2$ is
a limit point of the corresponding sequence of points 
$\left(\frac{z_j}{p^{N_j}},\frac{f(z_j)\md{p^{N_j}}}{p^{N_j}}\right)$
as $j\to\infty$. This means that $\Cal E(f)=\mathbb I^2$ and thus $\alpha(f)=1$.

\textbf{Case 2}: 
No point from $\Cal E(f)$ has
an open neighbourhood that lies completely in $\Cal E(f)$; i.e., any open neighbourhood
$U$ of any point from $\Cal E(f)$ contains points from the  subset $\mathbb I^2\setminus\Cal E(f)$, which is open in $\mathbb I^2$. 

Hence, $U$ contains an open subset that lies completely in $\mathbb I^2\setminus\Cal E(f)$ (we assume that $\mathbb I^2\setminus\Cal E(f)\ne\emptyset$ since otherwise
$\alpha(f)=1$ and there is nothing to prove).
Then there exists an open square 
$$
T_m(a,b)=\left(\frac{a}{p^{m}};\frac{a}{p^{m}}+\frac{1}{p^{m}}\right)\times
\left(\frac{b}{p^{m}};\frac{b}{p^{m}}+\frac{1}{p^{m}}\right),
$$
where $a,b\in\{0,1,\ldots,p^m-1\}$, that lies completely in $\mathbb I^2\setminus\Cal E(f)$. That is, $T_m(a,b)$ contains no points of the form 
$$
\left(\frac{x\md p^k}{p^{k}},\frac{f(x)\md p^k}{p^{k}}\right),
$$
where $x\in\Z_p$ and $k\in\N$.

In other words this means that there exist words $\tilde a,\tilde b$ of length $m$ in the
alphabet $\F_p$ (which are just base-$p$ representations of
$a$ and $b$, respectively) such that, whenever the automaton $\mathfrak
A=\mathfrak A_f$ is feeded by any input word $\tilde w$  with suffix $\tilde a$, i.e., 
$w=p^{\ell+m}a+u$ where $u\in\{0,1,\ldots,p^{\ell}-1\}$, the
corresponding output word $f(w)\md{ p^{\ell+m}}=p^{\ell+m}t+v$, 
$v\in\{0,1,\ldots,p^{\ell}-1\}$, never
has the suffix $\tilde b$, i.e., $t\ne b$ for all $\ell\in\N_0$ and all 
$u\in\{0,1,\ldots,p^{\ell}-1\}$ ($u$ is the empty word if $\ell=0$).

It is clear now that given any numbers $a^\prime, b^\prime\in\{0,1,\ldots,
p^{m^\prime}-1\}$, $m^\prime\ge m$,  such that
$a^\prime\equiv a\pmod{p^m}$, $b^\prime\equiv b\pmod{p^m}$, the corresponding
open square $T_{m^\prime}(a^\prime,b^\prime)$ 
lies completely outside of
$\Cal E(f)$, i.e., contains no points of the form $\left(\frac{x\md p^k}{p^{k}},\frac{f(x)\md p^k}{p^{k}}\right)$, where $x\in\Z_p$ and $k\in\N$. Indeed, otherwise some input word $w^\prime$ with  the suffix $a^\prime$ results in the output word
with the suffix $b^\prime$;  
but, this means that
the corresponding initial subword (whose suffix is $a$) of the word $w^\prime$
results in output word whose suffix is $b$. The latter case contradicts our
choice of $a,b$.

Now take $m^\prime=im$ for $i=1,2,\ldots$ and construct inductively a collection $\Cal T_i$ that consists of $(p^{2m}-1)^{i-1}$ disjoint open squares $T_{m^\prime}(a^\prime,b^\prime)$.
The collection $\Cal T_1$ consists of the only square $T_m(a,b)$.

Given the collection $\Cal T_{i-1}$, the collection $\Cal T_{i}$ consists of all open
squares $T_{im}(a^\prime,b^\prime)$, where
$a^\prime,b^\prime\in\{0,1,\ldots,p^{im}-1\}$, $a^\prime\equiv a\pmod{p^m}$,
$b^\prime\equiv b\pmod{p^m}$, that are disjoint from all squares from the collections
$\Cal T_1,\ldots\Cal T_{i-1}$. 

That is, at the first step we obtain a collection $\Cal T_1$ that consists of
the only $p^{-m}\times p^{-m}$ square $T_1(a,b)$; 
on the second step we obtain a collection $\Cal T_2$ that consists
of $p^{2m}-1$ disjoint $p^{-2m}\times p^{-2m}$-squares; on the third step
we obtain a collection $\Cal T_3$ that consists of $(p^{2m}-1)p^{2m}-(p^{2m}-1)=(p^{2m}-1)^2$
disjoint $p^{3m}\times p^{3m}$-squares, etc.

The union $T$ of all these open squares from $\Cal T_1, \Cal T_2, \ldots$ is
open, whence, measurable, and the Lebesgue measure of $T$ is
$$
\frac{1}{p^{2m}}+ (p^{2m}-1)\cdot\frac{1}{p^{4m}}+(p^{2m}-1)^2\cdot\frac{1}{p^{6m}}+\cdots=1
$$
since all these open squares are disjoint by the construction.
On the other hand, by the construction $T$ contains no points of the form $\left(\frac{x\md p^k}{p^{k}},\frac{f(x)\md p^k}{p^{k}}\right)$, where $x\in\Z_p$ and $k\in\N$.
Consequently, $T\cap\Cal E(f)=\emptyset$; in turn, this implies that the
Lebesgue measure of $\Cal E(f)$ must be 0, i.e, that $\alpha(f)=0$. The
latter contradicts the  assumption  from the beginning of the proof.
This proves  the theorem. 
\end{proof}

\subsection{Completely transitive automata}
From Theorem \ref{thm:Auto_0-1} we immediately derive the second main result of the paper:
\begin{thm}[Criterion of complete transitivity]
\label{thm:comp-trans}
A system $\mathfrak A$ 
is completely transitive  if and only if $\alpha(f_{\mathfrak A(s_0)})=1$.
\end{thm}
\begin{proof}
Follows from Theorem \ref{thm:Auto_0-1}, cf.
equivalent definition of complete transitivity in terms of words.
\end{proof}
\begin{note}
Nowhere in the proofs of Theorem \ref{thm:Auto_0-1} and of Theorem \ref{thm:comp-trans}
we used that $p$ is a prime; so both theorems are true without this limitation.
\end{note}
A finite system (i.e., the one whose set of states is finite) can be word transitive; 
the odometer 
$ x\mapsto x+1$ on $\Z_2$ serves
as an example.
On the other hand, by \cite[Theorem 11.10]{AnKhr}, 
given a finite system  $\mathfrak A$,
the set $\Cal E(f_{\mathfrak A})$ is nowhere dense; so from Theorems \ref{thm:Auto_0-1} and \ref{thm:comp-trans} it follows that a finite
system
can not be completely transitive. Thus, $\alpha (\mathfrak A(s_0))=0$ if $\mathfrak A(s_0)$ is a finite automaton. 

To construct automata $\mathfrak A$ of measure 1 (i.e., such that $\alpha(f_{\mathfrak
A})=1$) the following theorem (which is the third main result of the paper)
may be applied:
\begin{thm}[Sufficient conditions for complete transitivity]
\label{thm:comp-trans-der}
Let $f=f_{\mathfrak A}\colon\Z_p\>\Z_p$ be the automaton function of an automaton
$\mathfrak A$, and let $f$ be differentiable
everywhere in a ball $B\subset\Z_p$ of a non-zero
radius. The function $f$ is of  measure 1 whenever the following two
conditions hold simultaneously:
\begin{enumerate}
\item $f(B\cap\N_0)\subset\N_0$;
\item  $f$ is two times differentiable at some point $v\in B\cap\N_0$, and
$f^{\prime\prime}(v)\ne 0$. 
\end{enumerate}
\end{thm}
\begin{proof}
We will show that for every sufficiently large $k$ and every $z,u\in\{0,1,\ldots,
p^k-1\}$ there exists $M=M(k)$ and  $a\in\{0,1,\ldots,p^M-1\}$ such that 
\begin{equation}
\label{eq:no_lac_00}
\left|\frac{a}{p^M}-\frac{u}{p^k}\right|<\frac{1}{p^k}\ \text{and}\
\left|\frac{f(a)\md p^M}{p^M}-\frac{z}{p^k}\right|<\frac{1}{p^k}.
\end{equation} 
This will prove the theorem  
as every point from the unit square $\mathbb I^2$ can
be approximated by points of the form $\left(\frac{u}{p^k},\frac{z}{p^k}\right)$.

Briefly, our idea of the proof is as follows: 
As $v\in\N_0$, there exists $k\in\N_0$ such that all terms $\nu_i\in\{0,1,\ldots,p-1\}$ in the $p$-adic expansion
$v=\sum_{i=0}^\infty\nu_i\cdot p^i$ are zero, for all $i\ge k$.
We then somehow tweak $v$: Namely, we replace zeros in the $p$-adic expansion
at positions starting
with $\ell$-th, $\ell>k$, by certain other letters from $\{0,1,\ldots,p-1\}$
so that the tweaked $v$, the natural number $a=v+p^\ell t$,  satisfies inequalities
\eqref{eq:no_lac_00} for some $M$. 

As $f$ is differentiable everywhere in $B$, for $x\in B$ we have that given arbitrary $K\in\N$, the following congruence
holds for all $h\in\Z_p$ and all 
sufficiently large $L\in\N$: 
\begin{equation}
\label{eq:der}
f(x+p^L h)\equiv f(x)+p^L h\cdot f^\prime(x)\pmod{p^{K+L}}. 
\end{equation} 
Indeed, given $a,b\in\Z_p$, the condition $\|a-b\|_p\le p^{-d}$ is equivalent
to the condition $a\md{p^d}=b\md{p^d}$, where $\md{p^d}$ is a reduction modulo
$p^d$, cf. \eqref{eq:md=ineq}; so \eqref{eq:der} is just re-statement of a condition of differentiability
of a function at a point, in terms of congruences rather than in terms of
inequalities for $p$-adic absolute values: we just write $a\equiv b\pmod{p^d}$
instead of $\|a-b\|_p\le p^{-d}$.

Let $\|f^{\prime\prime}(v)\|_p=p^{-s}$; that is, 
$f^{\prime\prime}(v)=p^s\cdot\xi$, where $s\in\Z$ and
$\xi$ is a unity of $\Z_p$ (in other words, $\xi$ has a multiplicative inverse
in $\Z_p$). 
Note that  $s$ is not necessarily
non-negative since $f^{\prime\prime}(v)$ is in $\Q_p$, and not necessarily
in $\Z_p$; nonetheless further in the proof we assume that $k+s>0$ as we
may take $k$ large enough. Remind that $\|f^\prime(x)\|_p \le 1$
as $f$ is 1-Lipschitz; so $f^\prime(x)\in\Z_p$.

Now let $r\in\N$ be an arbitrary number such that $r>s$, $p^r>v$, and $p^{-r}$ is less than
the radius of the ball $B$ (it is clear that there are infinitely many choices
of $r$). 
Given $r$, consider $n\in\N$ such that
$n>\max\{\log_p f(v+p^{k+r}t)\colon t=0,1,2,\ldots,p^k-1\}$ and
$n>2k+2r+2s$ (we remind that in view of condition
2 of the theorem, all $f(v+p^{k+r}t)$ are in $\N_0$ due to our choice of
$n$). Put
\begin{align}
\label{eq:no_lac_1}
\tilde u&=1+p^{k+r+s}u\\
\label{eq:no_lac_2}
\tilde z&=f^\prime(v)+p^{k+r+s}\hat z,
\end{align}
where $\hat z\in\{0,1,\ldots,p^k-1\}$  is such that $\lfloor\frac{\tilde z}{p^{k+r+s}}\rfloor\md p^k=z$. In other words, we choose  $\hat z$ in
such a way that the number whose
base-$p$ expansion stands in positions from $(k+r+s)$-th to $(2k+r+s-1)$-th
in the canonical $p$-adic expansion of $\tilde z$, is equal to $z$. Obviously,
given $f^\prime(v)$ and $z$, there exists a unique $\hat z$ that satisfies
this condition: $\hat z\equiv z-\lfloor\frac{f^\prime(v)}{p^{k+r+s}}\rfloor\pmod{p^k}$;
so
\begin{equation}
\label{eq:no_lac_more}
\tilde z\md p^{2k+r+s}=(f^\prime(v)\md p^{k+r+s}) +p^{k+r+s}\cdot z.
\end{equation}

As $f$ is two times differentiable at $v$, for every $\zeta\in\{0,1,\ldots p^k-1\}$ 
we conclude that
\begin{equation}
\label{eq:no_lac_0}
f^\prime(v+p^{r+k}\zeta)\equiv f^\prime(v)+p^{r+k}\zeta\cdot f^{\prime\prime}(v)
\pmod{p^{2k+r+s}}
\end{equation}
for all sufficiently large $r$ (formally, we just substitute $f^\prime$ for
$f$, $v$ for $x$,
$\zeta$ for $h$, $k+s$ for $K$, and $r+k$ for $L$
in  \eqref{eq:der}). 
From here we deduce that as $f$ is  differentiable in $B$,  the following congruence holds for all sufficiently large $n$:
\begin{multline}
\label{eq:no_lac_3}
f(v+p^{r+k}\zeta+p^n\tilde u)\equiv f(v+p^{r+k}\zeta)+\\
p^n\tilde u\cdot(f^\prime(v)+p^{r+k}\zeta\cdot f^{\prime\prime}(v))\pmod{p^{n+2k+r+s}}.
\end{multline}
Note that the latter congruence is obtained by combination of congruence \eqref{eq:der}
where $K=2k+r+s$, $x=v+p^{r+k}\zeta$, $h=\tilde u$ and $L=n$, with congruence \eqref{eq:no_lac_0}. 

We claim that there exists $\zeta\in\{0,1,\ldots p^k-1\}$ such that
\begin{equation}
\label{eq:no_lac}
\tilde u\cdot(f^\prime(v)+p^{r+k}\zeta\cdot f^{\prime\prime}(v))
\equiv \tilde z\pmod{p^{2k+r+s}}.
\end{equation}
Indeed, in view of \eqref{eq:no_lac_1}--\eqref{eq:no_lac_2} this congruence is equivalent to the congruence
$(1+p^{k+r+s}u)\cdot (f^\prime(v)+p^{r+k}\zeta\cdot f^{\prime\prime}(v))
\equiv f^\prime(v)+p^{k+r+s}\hat z\pmod{p^{2k+r+s}}$, and the latter congruence
is equivalent to the congruence
$f^\prime(v)+p^{r+k}\zeta\cdot f^{\prime\prime}(v)\equiv
(1-p^{k+r+s}u)\cdot(f^\prime(v)+p^{k+r+s}\hat z)\pmod{p^{2k+r+s}}$ as 
$(1+p^{k+r+s}u)^{-1}\equiv 1-p^{k+r+s}u\pmod{p^{2k+r+s}}$. That is, congruence
\eqref{eq:no_lac} is equivalent to the congruence 
$p^{k+r}\zeta\cdot f^{\prime\prime}(v)\equiv p^{k+r+s}\hat z-p^{k+r+s}u\cdot
f^\prime(v)\pmod{p^{2k+r+s}}$. Further, as $f^{\prime\prime}(v)=p^s\xi$,
the latter congruence is equivalent to the congruence
$\zeta\xi\equiv \hat z-u\cdot f^\prime(v)\pmod{p^k}$. From here we find 
$\zeta\equiv\xi^{-1}\cdot(\hat z-u\cdot f^\prime(v))\pmod{p^k}$, thus proving
our claim (we remind that  $\xi$ is a unity of $\Z_p$; hence, 
$\xi$ has a multiplicative inverse
$\xi^{-1}$ modulo $p^k$).

Now we put $M=n+2k+r+s$ and $a=v+p^{r+k}\zeta +p^n\cdot(1+p^{k+r+s}u)$; then
$$
\frac{a}{p^M}=\frac{u}{p^k}+\frac{v+p^{r+k}\zeta+p^n}{p^{n+2k+r+s}},
$$
so $\left|\frac{a}{p^M}-\frac{u}{p^k}\right|<\frac{1}{p^k}$, since $v<p^r$, $\zeta<p^k$,
and $n>2r+2s+2k$. Also, combining \eqref{eq:no_lac},
\eqref{eq:no_lac_2}, \eqref{eq:no_lac_more}, and \eqref{eq:no_lac_3}, we see that
\begin{equation}
\label{eq:no_lac_4}
\frac{f(a)\md p^M}{p^M}=\frac{z}{p^k}+\frac{f(v+p^{r+k}\zeta)}{p^n}\cdot
\frac{1}{p^{2k+r+s}}+\frac{f^\prime(v)\md p^{k+r+s}}{p^{k+r+s}}\cdot\frac{1}{p^{k}},
\end{equation}
since $f(a)\md p^M=f(v+p^{r+k}\zeta)+p^n\cdot (f^\prime(v)\md p^{k+r+s})+p^{n+k+r+s}z$
(the number in the right-hand side is less than $p^M$ due to our choice of
$n$). 
Now from \eqref{eq:no_lac_4} it follows that 
$\left|\frac{f(a)\md p^M}{p^M}-\frac{z}{p^k}\right|<\frac{1}{p^k}$ since
$0\le f(v+p^{r+k}\zeta)\le p^n-1$ 
due to our
choice of $n$. 
\end{proof}
\begin{note}
\label{note:minus}
We note that $\alpha(f(x))=\alpha(-f(x))=\alpha(f(-x))$ for every 1-Lipschitz
function $f\colon\Z_p\>\Z_p$ of variable $x$; so we may replace condition
1 of Theorem \ref{thm:comp-trans-der} by either of conditions $f(B\cap-\N_0)\subset\N_0$,
$f(B\cap\N_0)\subset-\N_0$, or $f(B\cap-\N_0)\subset-\N_0$, where $-\N_0=\{0,-1,-2,\ldots\}$.

Indeed, for every $c\in\N$ and every
$n\in\N$ we have that $\frac{-c\md p^n}{p^n}=\frac{p^n-(c\md p^n)}{p^n}=
1-\frac{c\md p^n}{p^n}$. Thus, a symmetry with respect to the axis 
$y=\frac{1}{2}$ of the unit square $\mathbb I^2\subset\R^2$ 
maps the subset 
$$
E(f)=\left\{\left(\frac{x\md p^n}{p^n},\frac{f(x)\md p^n}{p^n}\right)\colon x\in\Z_p,
n\in\N\right\}\subset\mathbb I^2
$$ 
onto the subset $E(-f)$ and vice versa; so $\alpha(f(x))=\alpha(-f(x))$. A similar argument proves that $\alpha(f(x))=\alpha(f(-x))$.
\end{note}
By using Theorem \ref{thm:comp-trans-der}, one may construct numerous 
automata (and systems) that are completely transitive.
For instance, given $c\in\{2,3,4,\ldots\}$, listed below are examples of automata functions $f_{\mathfrak A(s_0)}=f$
that satisfy Theorem \ref{thm:comp-trans-der}; so the corresponding
 automata  $\mathfrak A(s_0)$ are completely transitive:
\begin{itemize}
\item $f(x)=cx+c^x$ if  $c\ne 1$, $c\equiv 1\pmod p$;
\item $f(x)=(x\AND c)+((x^2)\OR c)$ if $p=2$.
\end{itemize} 
Note that the
first of these automata is word transitive  while the second one is not.

With the use of Theorem
\ref{thm:comp-trans-der} new types of absolutely transitive automata can
be constructed as well. 
The following corollary from Theorem \ref{thm:comp-trans-der} is a key tool
in the construction of these:
\begin{cor}
\label{cor:abs-t-diff}
Let an automaton function $f=f_{\mathfrak A(s_0)}$ map $\N_0$ into $\N_0$,
let $f$ be two times differentiable on $\Z_p$, and let
$f^{\prime\prime}(x)$ have no more than a finite number of zeros in $\N_0$.
Then the automaton $\mathfrak A(s_0)$ is absolutely transitive.
\end{cor} 
\begin{proof}
Given  a finite non-empty word $\tilde g$ (say, of length $m>0$) over the
alphabet $\F_p$, take a finite
word $\tilde v$ whose prefix is $\tilde g$ and
such that the corresponding non-negative rational integer $v$~\footnote{The
one whose base $p$-expansion is $\tilde v$; remind that according to our
conventions words are read from
right to left, that is the rightmost letters of $\tilde v$ correspond to low order digits in the base-$p$ expansion of $v$.} 
is a non-zero of $f^{\prime\prime}$:
$f^{\prime\prime}(v)\ne 0$. The word $\tilde v$ that satisfies 
these conditions simultaneously 
exists
as $f^{\prime\prime}$ has not more than a finite number of zeros in $\N_0$
(fixing arbitrary $\tilde g$ means that only some less
significant digits in the base-$p$ expansion of $v$ are fixed);
so by taking $v$ whose base-$p$ expansion 
is sufficiently long (thus making $v$ large enough), we find $v\in\N_0$ such that $f^{\prime\prime}(v)\ne
0$ and the $m$-letter prefix of the word $\tilde v$ is $\tilde g$. 

In other
words, given an arbitrary finite word $\tilde g$ 
over the alphabet $\F_p$,  by properly choosing $r\in\N_0$ we find a positive
rational integer $v=g+p^m r$ (where $g\in\{0,1,\ldots,p^m-1\}$, $\tilde g$ is
a base-$p$ expansion of $g$)  such that $f^{\prime\prime}(v)\ne 0$. This
is possible due to the finiteness of a number of zeros of $f^{\prime\prime}$
in $\N_0$.
We see that both $f$ and the so constructed $v$ satisfy conditions of Theorem \ref{thm:comp-trans-der}: just assume that the ball $B$ from the conditions
of the mentioned theorem is the whole space $\Z_p$.

Now note that the claim stated at the very beginning of the proof of Theorem
\ref{thm:comp-trans-der} is just a re-statement of (ii) from Definition 
\ref{def:auto-trans-e}: Indeed, under notation of Definition \ref{def:auto-trans-e}
and the one from the beginning of the  proof of Theorem \ref{thm:comp-trans-der}, the concatenation $w\circ
y$ corresponds to $a$, $w$ corresponds to $u$, $w^\prime$ corresponds to $z$, and
$w$ is a $k$-letter
suffix of the output word which is a  base-$p$ expansion for $f(a)\md{p^M}$, whereas
$M$ is the length of the word $w\circ y$. Up to these correspondences, condition \eqref{eq:no_lac_00}
is equivalent to (ii) from  Definition \ref{def:auto-trans-e}.
Furthermore,  as the word $\tilde v$ has an arbitrarily chosen
prefix $\tilde g$, and as the condition \eqref{eq:no_lac_00} holds for  $a=v+p^\ell t$ from the proof of Theorem \ref{thm:comp-trans-der}
(as the whole
Theorem \ref{thm:comp-trans-der} holds for $f$ and $v$), 
(ii)
from Definition \ref{def:auto-trans-e} holds for input word with arbitrarily
chosen prefix $\tilde g$, up to all  mentioned correspondences. This means that  
(iii) from Definition \ref{def:auto-trans-e} also holds 
for $x=\tilde g$ in the case under consideration. The latter finally proves Corollary
\ref{cor:abs-t-diff}.
\end{proof}
We remark that Note \ref{note:minus} can be applied to Corollary \ref{cor:abs-t-diff}
as well.

Note also that the only type of absolutely transitive automata  $\mathfrak A(s_0)$ were known earlier: The ones whose automata functions are polynomials over $\Z$  of degree greater than
1,  
see \cite[Theorem 11.11]{AnKhr}. The latter assertion follows from Corollary
\ref{cor:abs-t-diff}. Yet many other types automata can be proved to be
absolutely transitive as well by using the corollary. For instance, an automaton whose both
input and output alphabets are $\F_2$ and whose automata function is $f(x)=a+bx+((x^2)\OR
c)$, where $a,b,c\in\N_0$, is absolutely transitive: This easily follows 
from Corollary
\ref{cor:abs-t-diff} as  $f(\N_0)\subset\N_0$ and $f^{\prime\prime}(x)=2$
for all $x\in\Z_2$.

\section{Discussion}
\label{sec:Concl}
In the paper, by combining tools from $p$-adic and real analysis and automata
theory we have shown that discrete systems (automata) with respect to the transitivity of their actions on finite
words constitute two classes,
the systems whose  real plots  have Lebesgue measures 1 (equivalently, the completely
transitive systems; i.e. such that given two arbitrary words $w$, $w^\prime$ of equal lengths, the system transforms
$w$ into $w^\prime$) and  systems whose real plots have Lebesgue measures 0.
Also we have found conditions for complete transitivity of a system; the conditions yield a method
to construct numerous completely transitive automata and
respective automata functions, especially
the ones that are combined from standard computer instructions and thus
are easily programmable.
The ergodic completely transitive automata are
preferable in constructions of various pseudo-random number generators aimed
at
cryptographic and/or simulation usage; e.g., in stream ciphers and quasi
Monte-Carlo methods.


\end{document}